\documentclass[12pt]{amsart}
\usepackage{enumerate}
\usepackage{amsthm}
\usepackage{amsmath,amssymb,latexsym,amscd}
\usepackage{tikz}
\usepackage[colorlinks=true, urlcolor=Blue, pdfborder={0 0 0}]{hyperref}
\makeatletter
\@namedef{subjclassname@2010}{%
	\textup{2020} Mathematics Subject Classification}
\makeatother
\usepackage[noadjust]{cite}
\theoremstyle{definition}
\newtheorem{theorem}{Theorem}[section]
\newtheorem{lemma}[theorem]{Lemma}
\newtheorem{proposition}[theorem]{Proposition}

\theoremstyle{definition}
\theoremstyle{definition}

\newtheorem{remark}[theorem]{Remark}
\newtheorem{example}[theorem]{Example}
\usepackage{indentfirst}

\begin{document}
	\baselineskip=17pt
	\title[]{Orbit spaces  of Free Involutions on The Product of Three Spheres}
	\author[Dimpi and Hemant Kumar Singh]{ Dimpi and Hemant Kumar Singh}
	\address{ Dimpi \newline 
		\indent Department of Mathematics\indent \newline\indent University of Delhi\newline\indent 
		Delhi -- 110007, India.}
	\email{dimpipaul2@gmail.com}
	\address{  Hemant Kumar Singh\newline\indent 
		Department of Mathematics\newline\indent University of Delhi\newline\indent 
		Delhi -- 110007, India.}
	\email{hemantksingh@maths.du.ac.in}

	\date{}
	\thanks{The first author of the paper is  supported by SRF of UGC, New Delhi, with  reference no.: 201610039267.}
	\begin{abstract} 
		
		\noindent In this paper, we have determined the orbit spaces of free involutions on  a finitistic space having mod $2$  cohomology of the product of three spheres $\mathbb{S}^n\times \mathbb{S}^m \times \mathbb{S}^l, 1 \leq n \leq m \leq l$. 	This paper generalizes the results proved by Dotzel et al. \cite{Dotzel} for free involutions on the product of two sphere $\mathbb{S}^n \times \mathbb{S}^m,1\leq n\leq m.$
		As an application, we have  also derived the Borsuk-Ulam type results. 
		
	\end{abstract}
	

	\maketitle
	\section {Introduction}

	\noindent  Let $G$ be a compact Lie group acting on a compact Hausdroff space $X.$ The study of orbit space $X/G,$ when $G$ acts freely on $X,$ has attracted many mathematicians over the world since the beginning of the twentieth century. In $1925$-$26,$ H. Hopf raised the question to classify the orbit space of free actions of finite cyclic groups on spheres. In general, it is difficult  to classify the orbit spaces up to homeomorphism or homotopy type. In this direction, Oliver\cite{oliver} proved  that the orbit spaces of actions of a compact lie group $G$ on   Euclidean spaces $\mathbb{R}^n$ are contractible. In $1964,$ Hirsch et al.  \cite{milnor} proved that the orbit space  of a free involution on $n$-sphere $\mathbb{S}^n$ is a homotopy type of $\mathbb{R}P^n.$ Further, the orbit spaces of finite group actions on $n$-sphere $\mathbb{S}^n$ have been studied in \cite{livesay,rice,ritter,rubin}. Su  \cite{su} determined the orbit spaces of free circle actions on spheres.  Tao \cite{tao} determined the orbit spaces of free involutions on $\mathbb{S}^1\times \mathbb{S}^2.$  Ritter  \cite{Ritter} extended the Tao's results to cyclic groups of order $2n.$ In 1972, Ozeki et al. \cite{ozeki} determined the orbit space of a
	free circle action on a manifold whose cohomology ring is isomorphic to the product of spheres $\mathbb{S}^{2n+1}\times \mathbb{S}^{2n+1}$ with integer coefficients. The orbit spaces
	of free actions of $G=\mathbb{Z}_p, p$ a prime, or $G=\mathbb{S}^d,d=1$ or $3,$ on a finitistic space $X$ whose mod $2$ or rational cohomology ring is isomorphic to the product of two spheres $\mathbb{S}^n \times \mathbb{S}^m$ have been studied in \cite{Dotzel} and \cite{anju}. However, the results for compact manifolds other than spheres  are less known.  Myers \cite{myner}
	determined the  orbit spaces of free involutions on three dimensional lens spaces. Recently, Dey et al. \cite{dey}, Morita et al. \cite{morita}, and  Singh \cite{singh} discussed the orbit spaces of free involutions on real Milnor manifolds, Dold manifolds,
	and the product of two projective spaces, respectively.

	Continuing this thread of research, this paper is concerned with the study of the orbit spaces of free involutions on a finitistic space having mod $2$ cohomology of the product of three spheres $\mathbb{S}^n \times  \mathbb{S}^m \times \mathbb{S}^l,$ $n \leq m \leq l .$  For example: (1) The complex Stiefel manifold $V_{n,n-3}$ has integral cohomology of the product of three spheres $\mathbb{S}^{2n-9}\times \mathbb{S}^{2n-7}\times\mathbb{S}^{2n-5},$ for all $n\geq 5$ \cite{adam}, which admits a free involution defined by $(v_1,v_2,\cdots, v_{n-3})\mapsto (gv_1,gv_2,\cdots, gv_{n-3}), $ where $v_i's, 1\leq i \leq n-3,$ are orthonormal vectors in $\mathbb{C}^n$ and $g \in \mathbb{Z}_2,$ and (2) The product of special unitary 3-group and a sphere  $SU(3)\times \mathbb{S}^l$ and unitary 2-group and a sphere $U(2)\times \mathbb{S}^l$ has integral cohomology of the product of three spheres $\mathbb{S}^3 \times  \mathbb{S}^5 \times \mathbb{S}^l$ and $\mathbb{S}^1 \times  \mathbb{S}^3 \times \mathbb{S}^l,$ respectively \cite{borel}. Consider,  diagonal actions on  $SU(3)\times \mathbb{S}^l$ and $U(2)\times \mathbb{S}^l$ obtained by taking trivial actions of $G=\mathbb{Z}_2$ on $SU(3)$ and $U(2),$ and the antipodal action on $\mathbb{S}^l.$ This gives  free involutions   on  $SU(3)\times \mathbb{S}^l$ and $U(2)\times \mathbb{S}^l,$ respectively.
	This paper generalizes the results proved by Dotzel et al. \cite{Dotzel} for free involutions on the product of two sphere $\mathbb{S}^n \times \mathbb{S}^m,1\leq n\leq m.$
	As an application, we have also determined the  Borsuk-Ulam type results.

	\section{Preliminaries}  
	\noindent	In this section, we review some basic definitions and results that are used in this paper.
	A paracompact Hausdroff space is called finitistic if every open covering has a finite dimensional open refinement. This is the most suitable topological space for the study of the relationship between the cohomology structure of the total space and that of the orbit space of a transformation group. It includes all compact Hausdroff spaces and paracompact spaces with finite covering dimension. Let $G$ be a finite cyclic group  acting on a finitistic space $X.$  The associated Borel fibration is $ X \stackrel{i} \hookrightarrow X_G \stackrel{\pi} \rightarrow B_G,$ where $X_G = (E_G\times G)/G$ (Borel space) obtained by diagonal action of $G$ on space $X\times E_G$ and $B_G$ (classifying space) is the orbit space of  free action of $G$ on contractible space $E_G.$  We recall some results of Leray-Serre spectral sequence associated with Borel fibration $ X \stackrel{i} \hookrightarrow X_G \stackrel{\pi} \rightarrow B_G.$ For proofs, we refer \cite{bredon,mac}.
	
	\begin{proposition}(\cite{mac})
		Suppose that  the system of local coefficients on $B_G$ is simple.	Then, the  homomorphisms $i^*: H^*(X_G) \rightarrow H^*(X)$ and $\pi^*: H^*(B_G) \rightarrow H^*(X_G)$ are the edge homomorphisms, \begin{center}
			$ H^k(B_G)=E_2^{k,0}\rightarrow E_3^{k,0}\rightarrow \cdots E_k^{k,0}\rightarrow E_{k+1}^{k,0} = E_{\infty}^{k,0} \subset H^k(X_G),$ and $  H^i(X_G) \rightarrow E_{\infty}^{0,i} \hookrightarrow E_{l+1}^{0,i} \hookrightarrow E_{l}^{0,i} \hookrightarrow \cdots \hookrightarrow E_{2}^{0,i} \hookrightarrow E_2^{0,i} \hookrightarrow H^i(X),$ respectively.
		\end{center} 
	\end{proposition}
	
	\begin{proposition}(\cite{mac})\label{2.4}
		Let $G=\mathbb{Z}_2$ act on a finitistic space $X$ and $\pi_1(B _G)$ acts trivially on $H^*(X).$ Then, the system of local coefficients on $B_G$ is simple and $$E^{k,i}_2= H^{k}(B_G) \otimes  H^i(X), ~k,i \geq 0.$$
	\end{proposition}
	
	\begin{proposition}\label{nontrivial}(\cite{bredon})
		Let $G=\mathbb{Z}_2$ act on a finitistic space $X$ and  $\pi_1(B _G)$ acts nontrivially on $H^*(X).$  Then,  the $E_2$ term of the Leray-Serre spectral sequence of the fibration  $ X \stackrel{i} \hookrightarrow X_G \stackrel{\pi} \rightarrow B_G$ is given by
		\[
		E_2^{k,i}=
		\begin{cases}
			\text{ker $\tau$} & \mbox{for}~ ~k=0, \\
			\text{ker $\tau$/im $\sigma$} & \text{for } k>0,\\

		\end{cases}
		\]
		where $\tau = \sigma = 1+g^*,$ $g^*$ is induced by a generator $g$ of $G.$
	\end{proposition}
	\begin{proposition}(\cite{mac})\label{mac}
		Let $G=\mathbb{Z}_2$   act freely on a finitistic space $X.$ Then, the Borel space $X_G$ is homotopy equivalent to the orbit space $X/G.$
	\end{proposition}
	\begin{proposition}(\cite{bredon})\label{prop 4.5}
		Let $G=\mathbb{Z}_2$   act freely on a finitistic space $X.$ If $H^i(X;\mathbb{Z}_2)=0~ \forall~ i>n,$ then $H^i(X/G;\mathbb{Z}_2)=0~ \forall~ i>n.$ 
	\end{proposition}

	\begin{proposition}\label{prop 5}(\cite{bredon})
		Let $G=\mathbb{Z}_2$ act on a finitistic space $X.$  Suppose that $H^i(X; \mathbb{Z}_2)=0$ for all $i>2n$ and $H^{2n}(X;\mathbb{Z}_2)=\mathbb{Z}_2.$ If there exists an element $a \in H^{n}(X;\mathbb{Z}_2)$ such that $ag^*(a)\neq 0,$ where  $g$ be a generator of $G,$ then fixed point set is nonempty.
	\end{proposition}
	\begin{proposition}\label{prop 6}(\cite{bredon})
		Let $G=\mathbb{Z}_2$ act on a finitistic space $X.$ Then, for any $a \in H^n(X),$ the element  $ag^*(a)$ is permanent cocycle in spectral sequence of $X_G \rightarrow B_G.$
	\end{proposition}
	
	\noindent Recall  that 
	$H^*(\mathbb{S}^n \times \mathbb{S}^m \times \mathbb{S}^l; \mathbb{Z}_2)=\mathbb{Z}_2[a,b,c]/<{a^{2},b^2, c^2}>,$ where deg $a=n,$ deg $b=m$ and  deg $c=l, ~1 \leq n \leq m\leq l.$\\
	\noindent Throughout the paper,  
	$H^*(X)$ will denote the \v{C}ech cohomology of a space $X$ with coefficient group $G=\mathbb{Z}_2,$ and   $X\sim_2 Y,$  means $H^*(X;\mathbb{Z}_2 )\cong H^*(Y;\mathbb{Z}_2).$

	\section{The  Cohomology  Algebra of  The Orbit Space of  Free involutions on $ \mathbb{S}^n \times \mathbb{S}^m \times \mathbb{S}^l $}
	\noindent In this section, we determine the cohomology ring of free involutions on finitistic spaces $X$ whose cohomology ring is isomorphic to the product of three spheres, $X\sim_2 \mathbb{S}^n \times \mathbb{S}^m \times \mathbb{S}^l, $ where $n \leq  m\leq l.$ We first prove the following lemma.
	\begin{lemma}
		Let $G = \mathbb{Z}_2$ act freely on a finitistic space $X\sim_2 \mathbb{S}^n \times \mathbb{S}^m \times \mathbb{S}^l, $ where $n < m< l.$ Then, $\pi_1(B_G)$ acts trivially on $H^*(X).$ 
	\end{lemma}
	\begin{proof}
		For $l \neq m+n,$ obviously, $\pi_1(B_G)$ acts trivially on $H^*(X).$ Next, suppose $l=m+n.$ If $\pi_1(B_G)$ acts nontrivially on $H^*(X),$ then we must have  $g^*(c)=ab+c,$ where $g$ is the generator of $\pi_1(B_G).$ So, we get $cg^*(c)=abc \neq 0,$ which contradicts Proposition $\ref{prop 5}.$ Hence, our claim.
	\end{proof}
	Note that $G=\pi_1(B_G)$ may act nontrivially on $H^*(X)$ in the following two cases: (i) $X\sim_2 \mathbb{S}^n \times \mathbb{S}^m \times \mathbb{S}^m, n<m,$ and (ii) $X\sim_2 \mathbb{S}^n \times \mathbb{S}^n \times \mathbb{S}^l,n\leq l.$ First, we determine the cohomology algebra of free involutions on $X\sim_2 \mathbb{S}^n \times \mathbb{S}^m \times \mathbb{S}^l$ in the case when $\pi_1(B_G)$ acts nontrivially on $H^*(X).$ We have  proved the following Theorems.
	\begin{theorem}\label{thm 3.2}
		Let $G = \mathbb{Z}_2$ act freely on a finitistic space $X\sim_2 \mathbb{S}^n \times \mathbb{S}^m \times \mathbb{S}^m ,$ where $n < m.$ If $\pi_1(B_G)$ acts nontrivially on $H^*(X),$ then $H^*(X/G)$ is isomorphic to the following graded commutative algebra:
		$${\mathbb{Z}_2[x,y,w,z]}/{<x^{n+1},y^2+a_0z,w^2,z^2,yw+a_1x^nz,yz,wz,xy,xw>},$$ where deg $x=1,$ deg $y=m$, deg $w=m+n,$ deg $z=2m$ and $a_0,a_1 \in \mathbb{Z}_2.$
	\end{theorem}
	\begin{proof}
		It is clear that $\pi_1(B_G)$ acts nontrivially on $H^m(X).$ Let  $g$ be a generator of $\pi_1(B_G).$ We have exactly three possibilities of  nontrivial actions: (1) $g^*(b)=c~\&$ $g^*(c)=b,$ (2) $g^*(b)=b~\&$ $g^*(c)=b+c$ and (3) $g^*(b)=b+c~\&$ $g^*(c)=c.$ First, consider nontrivial action defined by $g^*(b)=c~\&~g^*(c)=b.$  By the natuality of cup product, we get $\pi_1(B_G)$ acts nontrivially on $H^{m+n}(X).$ Note that $\sigma=\tau=1+g^*.$ By Proposition \ref{nontrivial}, we get $E_2^{0,i}\cong \mathbb{Z}_2$ and  $E_2^{k,i}= 0~~~\forall~~k>0, ~~ i=m,m+n.$ For  $i\neq m,m+n$, $\pi_1(B_G)$ acts trivially on $H^i(X),$ and hence, $E_2^{k,i} \cong H^k(B_G) \otimes H^i(X).$  By Proposition \ref{prop 6}, $ 1 \otimes bc$ is a permanent cocycle. If $d_{n+1}(1\otimes a)=0,$ then atleast two lines survives to $E_\infty$ which is not possible. Therefore, we get $d_{n+1}(1 \otimes a )= t^{n+1}\otimes 1,$ and $d_{n+1}(1 \otimes abc)=t^{n+1}\otimes bc.$ Clearly, $d_r =0 ~~ \forall~~ r > n+1.$ Thus, $E_{n+2}^{*,*}\cong E_{\infty}^{*,*}.$ So, we have  $E_{\infty}^{p,q}\cong\mathbb{Z}_2$ where $0\leq p \leq n,q=0,2m; p=0,q=m,n+m.$ Consequently, the cohomology groups  are given by  
		\[
		H^k(X_G)\cong\bigoplus_{p+q=k}{E_{\infty}^{p,q}}=
		\begin{cases}
			\mathbb{Z}_2  & j \leq k \leq n+j, j =0,2m; k=m,m+n,\\
			0 & \text{otherwise.}
		\end{cases}
		\]
		The permanent cocycles $t\otimes 1,$ $b+c$, $ab+ac$ and  $1 \otimes bc$  are determine the elements $x \in E_{\infty}^{1,0},$ $u \in E_{\infty}^{0,m},$ $v\in E_{\infty}^{0,n+m}$ and $s \in E_{\infty}^{0,2m},$ respectively. Thus, the total  complex Tot$E_{\infty}^{*,*}$ is given by $${\mathbb{Z}_2[x,u,v,s]}/{<x^{n+1},u^2+a_0s,v^2,s^2,uv,us,vs,xu,xv>},$$
		where deg$~x=1,$ deg$~u=m,$ deg$~v=m+n,$ deg$~s=2m, $ and $a_0 \in \mathbb{Z}_2.$ Let $y \in H^{m}(X_G),w \in H^{m+n}(X_G)$ and $z \in H^{2m}(X_G)$ such that $i^*(y)=b+c,i^*(w)=ab+ac$ and $i^*(z)=bc,$ respectively. Clearly, we have $yw+a_1x^{n}z=0, a_1 \in \mathbb{Z}_2$.  By Proposition \ref{mac},  the graded commutative algebra $ H^*(X/G)$ given by
		$${\mathbb{Z}_2[x,y,w,z]}/{<x^{n+1},y^2+a_0z,{w}^2,z^2,yw+a_1x^nz,yz,wz,xy,xw>},$$
		where deg$~x=1,$ deg$~y=m,$ deg$ ~w=m+n,$ deg$~z=2m, $ and $a_0,a_1 \in \mathbb{Z}_2.$  
		For the other two  nontrivial actions of $\pi_1(B_G)$ on $H^m(X),$ we get the same cohomology algebra.
	\end{proof}
	\begin{remark}
		In the above theorem, if $a_0=a_1=0,$ then  $X/G \sim_2 (\mathbb{RP}^n \times S^{2m}) \vee \mathbb{S}^m \vee \mathbb{S}^{m+n}.$
	\end{remark}
	Now, for a finitistic space $X\sim_2 \mathbb{S}^n \times \mathbb{S}^n \times \mathbb{S}^l ,$  $n\leq l,$ we have following Theorem.
	\begin{theorem}{\label{thm3.4}}
		Let $G = \mathbb{Z}_2$ act freely on a finitistic space $X\sim_2 \mathbb{S}^n \times \mathbb{S}^n \times \mathbb{S}^l ,$ where $n\leq l.$ If $\pi_1(B_G)$ acts nontrivially on $H^*(X),$ then $H^*(X/G)$ is isomorphic to one of the following graded commutative algebras:
		\begin{enumerate}
			\item  $\mathbb{Z}_2[x,y,w,z]/<x^{2n+l+1},I_j,z^2,xy,zx,x^{l-2n+1}w>_{1\leq j \leq 4},$ \\where  deg $x=1,$ deg $y=n,$ deg $w=2n~\&$ deg $z=n+l;$
			$I_1=y^2+a_1x^{2n}+a_2w,I_2={w}^2+a_3x^{4n}+a_4x^{2n}w+a_5z,I_3=yw+a_6x^{3n}+a_7x^nw~\&~I_4=yz+a_8x^{2n+l},$ $a_i \in \mathbb{Z}_2, 1 \leq i \leq 8;$  $a_3=0$ if $l<2n,a_4=0$ if $l<4n, a_5=0$ if $l\neq 3n~\&$ $ a_7=0$ if $l<3n.$ 	
			\item  $\mathbb{Z}_2[x,y,w,z]/<x^{l+1},I_j,z^2,xy,zx>_{1\leq j \leq 4},$ \\where  deg $x=1,$ deg $y=n,$ deg $w=2n~\&$ deg $z=n+l;$ 
			$I_1=y^2+a_1x^{2n}+a_2w,I_2={w}^2+a_3x^{4n}+a_4x^{2n}w+a_5z,I_3=yw+a_6x^{3n}+a_7x^nw+a_8z~\&~I_4=yz+a_9x^{l}w,$ $a_i \in \mathbb{Z}_2, 1 \leq i \leq 9;$ $a_1=a_4=0$ if $l<2n,a_3=0$ if $l<4n, a_5=0$ if $l\neq 3n,$$ a_6=0$ if $l<3n~\&~ a_8=0$ if $l \neq 2n.$
			
			\item ${\mathbb{Z}_2[x,y,w,z]}/{<x^{n+1},I_j,w^2,z^2,wz,xz,xy>_{1\leq j \leq 3}},$ \\
			where deg $x=1,$ deg $y=n$ $\&$ deg $w$= deg $z=2n;$ $I_1=y^2+a_1w+a_2z,I_2=yw+a_3x^{n}w~\&~I_2=yz+a_4x^{n}w,$ $a_i \in \mathbb{Z}_2, 1 \leq i \leq 4.$ 
			
		\end{enumerate}
	\end{theorem} 
	\begin{proof}
		We consider two cases: (i) $n< l,$ and (ii) $n= l.$\\
		{\bf{Case (i):}} Assume $n< l.$\\
		It is clear that if $l\neq 2n,$ then  $\pi_1(B_G)$ must act nontrivially on $H^n(X).$ Now, if $l=2n$ and $\pi_1(B_G)$ acts nontrivially on $H^{2n}(X),$ then  we must have  $g^*(c)=c+ab.$ Thus, we get  $cg^*(c)=abc\neq 0.$ By Proposition \ref{prop 5}, the fixed point set is empty, a contradiction. Therefore, $\pi_1(B_G)$ always acts trivially on $H^{2n}(X).$ So, $\pi_1(B_G)$ must  act nontrivially on $H^{n}(X).$  As in Theorem \ref{thm 3.2}, we have  exactly three possibilities of  nontrivial actions on $H^n(X).$ Suppose that $g^*(a)=b$ and $g^*(b)=a.$ Consequently, we get  $g^*(ac)=bc$ and $g^*(bc)=ac,$ where $g$ is the generator of $\pi_1(B_G).$ Note that $\sigma=\tau=1+g^*.$ By Proposition \ref{nontrivial}, we have $E_2^{0,i}\cong \mathbb{Z}_2,$ ~ $i=n,n+l$ and  $E_2^{k,i}= 0~~~\forall~~k>0, ~~ i=n,n+l.$ For  $i\neq n,n+l$, $\pi_1(B_G)$ acts trivially on $H^i(X).$ So, for $l\neq 2n,$ $E_2^{k,i} \cong \mathbb{Z}_2~\forall~k\geq 0;i=0,l,2n,2n+l$ and  for $l= 2n,$ $E_2^{k,i} \cong \mathbb{Z}_2~\forall~k\geq 0,i=0~\&~4n$ and $E_2^{k,2n} \cong \mathbb{Z}_2\oplus \mathbb{Z}_2~\forall~k\geq 0.$ By Proposition $\ref{prop 6},$  $ 1 \otimes ag^*(a)=1 \otimes ab$ is a permanent cocycle. If $d_{n+1}$ is  nontrivial, then we get $d_{n+1}(a+b)=t^{n+1} \otimes 1.$ So, we have $0 = d_{n+1}( (a+b)(t \otimes 1))=(t^{n+1} \otimes 1)(t\otimes 1) =t^{n+2}\otimes 1,$ which is not possible. Thus, $d_{n+1}(a+b)=0.$ 
		Suppose $d_{l-2n+1}\neq 0.$ Then, we must have $2n<l$ and  $d_{l-2n+1}(1 \otimes c)=t^{l-2n+1}\otimes ab.$ We get $E_{l-2n+2}^{*,*}=E_{2n+l} ^{*,*}.$ As $G$ acts freely on $H^*(X),$ we must have $d_{2n+l+1}(1 \otimes abc)= t^{2n+l+1} \otimes 1.$ Thus,  $E_{2n+l+2}^{*,*}=E_{\infty} ^{*,*},$ and hence  $E_{\infty}^{p,q}\cong \mathbb{Z}_2,$ where $0\leq p \leq 2n+l,q=0; 0\leq p \leq l-2n,q=2n;p=0,q=n,n+l.$ Thus, the cohomology groups are given by  
		\[
		H^k(X_G)=
		\begin{cases}
			\mathbb{Z}_2  & 0 \leq k < 2n, k \neq n; l< k \leq 2n+l, k \neq n+l\\
			\mathbb{Z}_2 \oplus \mathbb{Z}_2 & 2n \leq k \leq l, k=n,n+l.
		\end{cases}
		\]
		The permanent cocycles $t\otimes 1,$ $a+b,$ $1 \otimes ab$ and $ c(a+b)$ determines the elemnets $x \in E_{\infty}^{1,0},$ $u \in E_{\infty}^{0,n},$ $v\in E_{\infty}^{0,2n}$ and $s \in E_{\infty}^{0,n+l},$ respectively. Thus, the total complex is given by $${\mathbb{Z}_2[x,u,v,s]}/{<x^{2n+l+1},u^2+\gamma_1v,v^2,s^2,uv,us,vs,xu,x^{l-2n+1}v,sx>},$$
		where deg$~x=1,$ deg$~u=n,$ deg$~v=2n,$ deg$~s=n+l, $ and $\gamma_1\in \mathbb{Z}_2.$ Let $y \in H^{n}(X_G),w \in H^{2n}(X_G)$ and $z \in H^{n+l}(X_G)$ such that $i^*(y)=a+b,i^*(w)=ab$ and $i^*(z)=c(a+b),$ respectively.  We have $I_1= y^2+a_1x^{2n}+a_2w=0,  I_2=w^{2}+a_3x^{4n}+a_4x^{2n}w+a_5z=0;$ $I_3=yw+a_6x^{3n}+a_7x^nw=0$ and $I_4=yz+a_8x^{2n+l}=0$, where $a_i \in \mathbb{Z}_2, 1 \leq i \leq 8;$ $a_3=0$ if $l<2n,$ $a_4=0$ if $l<4n,a_5=0$ if $l\neq 3n$ and   $a_7=0$ if $l<3n.$ 
		By Proposition \ref{mac},  the graded commutative algebra $H^*(X/G)$ is given by
		$${\mathbb{Z}_2[x,y,w,z]}/{<x^{2n+l+1},I_j,z^2,x^{l-2n+1}w,wz,xz,xy>_{1\leq j\leq 4}},$$
		where deg$~x=1,$ deg$~y=n,$ deg$~w=2n,$ and deg$~z=n+l.$ This realizes possibility (1). \\
		Now, suppose $d_{l-2n+1}=0.$ If  $d_{l+1}=0,$ then atleast two lines survive to infinity, which contradicts Proposition \ref{prop 4.5}. So, we have $d_{l+1}(1 \otimes c)=t^{l+1} \otimes 1,$ and  hence $d_{l+1}(1 \otimes abc)=t^{l+1} \otimes ab.$ So, the differentials $d_r=0 ~\forall$ $r > l+1,$ which implies that $E_{\infty}^{*,*}= E_{l+2}^{*,*}.$ Consequently, $E_{\infty}^{p,q}\cong \mathbb{Z}_2,$ where $0\leq p \leq l,q=0,2n; p=0,q=n,n+l.$  For $l<2n$, the cohomology groups $H^k(X_G) $ are given by  
		\[
		H^k(X_G) =
		\begin{cases}
			\mathbb{Z}_2  & j \leq k \leq l+j, j=0,2n~\mbox{and}~k \neq n, n+l\\
			\mathbb{Z}_2 \oplus \mathbb{Z}_2 & k=n,n+l,
		\end{cases}
		\] and, for $l\geq 2n,$ we have 
		\[
		H^k(X_G) =
		\begin{cases}
			\mathbb{Z}_2  & 0 \leq k \leq 2n, k \neq n; l < k \leq 2n+l, k \neq n+l\\
			\mathbb{Z}_2 \oplus \mathbb{Z}_2 & k=n,n+l, 2n \leq k \leq l.
		\end{cases}
		\]
		The permanent cocycles $t\otimes 1,$ $a+b$, $1 \otimes ab$ and $ c(a+b)$ determine the elements $x \in E_{\infty}^{1,0},$  $u \in E_{\infty}^{0,n}$,$v\in E_{\infty}^{0,2n}$ and $s\in E_{\infty}^{0,n+l},$ respectively. Thus, the total complex is given by $${\mathbb{Z}_2[x,u,v,s]}/{<x^{l+1},u^2+\gamma_1v,v^2,s^2,uv+\gamma_2s,us,vs,xu,xv>},$$
		where deg$~x=1,$ deg$~u=n,$ deg$~v=2n,$ deg$ ~s=n+l, $ and $\gamma_1,\gamma_2\in \mathbb{Z}_2,$  $\gamma_2=0$ if $l \neq 2n.$ Let $y \in H^{n}(X_G),w \in H^{2n}(X_G)$ and $z \in H^{n+l}(X_G)$  such that $i^*(y)=a+b,i^*(w)=ab$ and $i^*(z)=ac+cb,$ respectively. We have $I_1= y^2+a_1x^{2n}+a_2w=0,  I_2=w^{2}+a_3x^{4n}+a_4x^{2n}w+a_5z=0,$ $I_3=yw+a_6x^{3n}+a_7x^nw+a_8z=0$ and $I_4=yz+a_9x^{l}w=0$, where $a_i \in \mathbb{Z}_2, 1 \leq i \leq 9;$  $a_1=a_4=0$ if $l<2n, a_3=0$ if $l<4n;a_5=0$ if $l\neq 3n,a_6=0$ if $l< 3n,$ and $ a_8=0$ if $l \neq 2n$. Therefore,  the graded commutative algebra $H^*(X/G)$ is given by
		$${\mathbb{Z}_2[x,y,w,z]}/{<x^{l+1},I_j,z^2,wz,xy,xw>_{1\leq j \leq 4}},$$
		where deg$~x=1,$ deg$~y=n,$ deg$~w=2n,$ and deg$~z=n+l.$ This realizes possibility (2). \\
		{\bf Case (ii):} Assume $n=l.$\\
		As $\pi_1(B_G)$  acts nontrivially on $H^n(X),$ $g^*$ must fixes one or two generator(s) of $H^n(X).$ If $g^*$ fixes one generator, say, $g^*(a)=a,$ then in this case three nontrivial actions are possible: $g^*(b)=c,g^*(c)=b; g^*(b)=a+b,g^*(c)=a+c;g^*(b)=a+c,g^*(c)=a+b.$ Further, if $g^*$ fixes two generators, say, $g^*(a)=a,g^*(b)=b,$ then also in this case, three nontrivial actions are possible: $g^*(c)=a+c;g^*(c)=b+c;g^*(c)=a+b+c.$ This gives eighteen different possibilities of nontrivial actions of $\pi_1(B_G)$ on $H^n(X).$ Now, consider a nontrivial action define as $g^*(a)=a,g^*(b)=c$ and $g^*(c)=b.$  So, we have   $g^{*}(ab)=ac, g^{*}(bc)=bc$ and $g^{*}(ac)=ab.$ By Proposition \ref{prop 5}, the $E_2$-page is given by 	\[
		E_2^{k,i}=
		\begin{cases}
			\text{ker $\tau$} & \mbox{for}~ ~k=0, \\
			\text{ker $\tau$/im $\tau$} & \text{for } k>0.\\

		\end{cases}
		\]
		Note that, for  $i=n,$  ker $\tau = <a,b+c>$ and im $\tau = <b+c>,$  and for $i=2n,$  ker $\tau = <bc,ab+ac)>$ and im $\tau = <ab+ac>.$ So, for $i=n,2n,$ we get $E_2^{0,i}\cong \mathbb{Z}_2 \oplus \mathbb{Z}_2,$  $E_2^{k,i}\cong \mathbb{Z}_2~\forall~ k>0.$  And,  for $i=0,3n,$ $\pi_1(B_G)$ acts trivially on $H^i(X).$ So, we get $E_2^{k,i}\cong H^k(B_G)\otimes H^i(X) \cong \mathbb{Z}_2~\forall~ k\geq 0.$  By  Proposition $\ref{prop 6},$ $bg^*(b)=bc$ is permanent cocycle. It is easy to observed that $d_{n+1}(a)$ must be nonzero. Therefore, $d_{n+1}(a)=t^{n+1}\otimes 1.$ Assume that $d_{n+1}(b+c)= a_0 t^{n+1}\otimes 1,$ $a_0 \in \mathbb{Z}_2.$ Consequently, $d_{n+1}(1\otimes abc)=(t^{n+1}\otimes 1)bc.$  So, we have $E_{n+2}^{*,*}=E_{\infty}^{*,*},$ and hence   $E_{\infty}^{p,q}\cong \mathbb{Z}_2,$ where $0\leq p \leq n,q=0;0<p\leq n,q=2n; p=0,q=n$ and $ E_{\infty}^{0,2n}\cong \mathbb{Z}_2 \oplus \mathbb{Z}_2.$  Thus, the cohomology groups of $X/G$ are given by 
		\[
		H^k(X_G)=
		\begin{cases}
			\mathbb{Z}_2  & 0 \leq k < n; 2n<k\leq 3n,\\
			\mathbb{Z}_2 \oplus \mathbb{Z}_2 &  k=n,2n.
		\end{cases}
		\]
		The	permanent cocycles $t\otimes 1,$ $c_0a+b+c,c_0\in \mathbb{Z}_2,$ $bc$ and $ ac+bc$  determine the elements $x \in E_{\infty}^{1,0}, u \in E_{\infty}^{0,n},$ $v\in E_{\infty}^{0,2n}$ and $s \in E_{\infty}^{0,2n},$ respectively.  Thus, the total complex Tot$E_{\infty}^{*,*}$ is given by $${\mathbb{Z}_2[x,u,v,s]}/{<x^{n+1},u^2+\gamma_1v+\gamma_2s,v^2,s^2,ux,uv,us,vs,sx>},$$
		where deg$~x=1,$ deg$~u=n,$ deg$~v=$ deg$~s=2n, $ and $\gamma_1 ,\gamma_2\in \mathbb{Z}_2.$ Let $y \in H^{n}(X_G),w \in H^{2n}(X_G)$ and $z \in H^{2n}(X_G)$  such that $i^*(y)=c_0a+b+c,i^*(w)=bc$ and $i^*(z)=ab+ac,$ respectively. Clearly, we have $I_1=y^2+a_1w+a_2z=0,$
		$I_2 = yw+a_3x^nw=0,$ and $I_3 = yz+a_4x^nw=0; a_i \in \mathbb{Z}_2, 1\leq i \leq 4.$  Therefore,  the graded commutative algebra $H^*(X/G)$ is given by
		$${\mathbb{Z}_2[x,y,w,z]}/{<x^{n+1},I_j,w^2,z^2,wz,xz,xy>}_{1\leq j\leq 3},$$
		where deg$~x=1,$ deg$~y=n,$ deg$~w=2n,$ and deg$~z=2n.$ This realizes possibility (3). \\
		For the other possibilities of nontrivial actions of $\pi_1(B_G)$ on $H^n(X),$ we get the same cohomology algebras as in case (i) and case (ii).
	\end{proof}
	
	\begin{remark}
		In the possibility $(2)$ of  Theorem \ref{thm3.4}, if we take $ a_i=0$ $\forall ~~ 1 \leq i \leq 9,$ then $X/G \sim_2 \mathbb{RP}^l \times S^{2n} \vee (\mathbb{S}^n \vee \mathbb{S}^{n+l}),$  and if we take $a_i=0$ $\forall ~~ 1 \leq i \leq 4$ in possibility $(3),$ then $X/G \sim_2 \mathbb{RP}^n \times S^{2n} \vee (\mathbb{S}^n \vee \mathbb{S}^{2n}).$ 
	\end{remark}
	Next, we determine the orbit spaces of free involutions on $X \sim_2 \mathbb{S}^n \times \mathbb{S}^m \times \mathbb{S}^l, $ when $\pi_1(B_G)$ acts trivially on $H^*(X).$\\

	\noindent	Recall that if $X \sim_2 \mathbb{S}^n \times \mathbb{S}^m \times \mathbb{S}^l ,$  $n< m< l,$ then $\pi_1(B_G)$ always acts trivially on $H^*(X).$ For the remaining cases, we assume that $\pi_1(B_G)$ acts trivially on $H^*(X).$ Let $\{E_r^{*,*},d_r\}$ be the Leray-Serre spectral sequence of the Borel fibration $X \hookrightarrow X_G \rightarrow B_G.$ So, by Proposition \ref{2.4}, we get  $$E_2^{k,i}=H^k(B_G) \otimes H^i(X)~\forall~  k,i\geq 0.$$ Thus, $E_2^{k,i}\cong \mathbb{Z}_2,$ for $k=0,n,m,l,n+m,n+l,m+l,n+m+l,$ and $\forall~ i\geq 0.$ \\ If  $G=\mathbb{Z}_2$  acts freely on $X,$ then at least one of the images of $1\otimes a,$ $1\otimes b$ or $1\otimes c,$ where $a,b,c$ are generators of $H^*(X),$ under some differential $d_r$ must be nonzero.\\
	\
	Note that  \begin{itemize}
		\item if $d_{r_1}(1\otimes a) \neq 0,$ then $r_1=n+1,$ 
		\item if $d_{r_2}(1\otimes b) \neq 0,$ then $r_2=m-n+1$ or $m+1,$ and 
		\item if $d_{r_3}(1\otimes c) \neq 0,$ then $r_3=l-m-n+1,$  $l-m+1,$ $l-n+1$ or $l+1.$
	\end{itemize} 
	So, the following cases are  possible:
	\begin{enumerate}
		
		\item $d_{r_1}(1\otimes a)\neq 0,$
		\item $d_{r_1}(1\otimes a)=0 $ and $d_{r_2}(1\otimes b) \neq 0,$ 
		\item $d_{r_1}(1\otimes a)=0 ,$ $d_{r_2}(1\otimes b) = 0$ and  $d_{r_3}(1\otimes c) \neq 0 .$
	\end{enumerate}
	We have discussed following theorems depending on above three cases:
	
	\begin{theorem}\label{thm 3.6}
		Let $G=\mathbb{Z}_2$ act freely on a finitistic space $X \sim_2 \mathbb{S}^n \times \mathbb{S}^m \times \mathbb{S}^l, $ where $n\leq m \leq l.$ If $d_{r_1}(1\otimes a)\neq 0,$  then $H^*(X/G)$ is isomorphic to one of the following graded commutative algebras:
		\begin{enumerate}
			\item  $\mathbb{Z}_2[x,y,w,z]/<x^{n+1}, I_j, z^2>_{1\leq j \leq 5}$,\\ where 	deg $x=1,$ deg $y=m$, deg $w=l~\&$ deg $z=m+l;$ 
			$I_1=y^2+a_1z+a_2x^ny+a_3x^{2m-l}w,I_2=w^2+a_4x^{l-m}z+a_5x^ny+a_6x^nw, I_3=yw+a_7z+a_8x^nw+a_9x^ny, I_4= yz+a_{10}x^nz,$ and $ I_5=wz+a_{11}x^nz,$ $a_i \in \mathbb{Z}_2$, $1\leq i \leq 11;$  $a_1=0$ if $n \leq m \neq l,$ $a_2=a_8=a_{10}=0$ if $n < m \leq l,$ $a_3=0$ if $2m-l > n $ or $l>2m,$ $a_4=0$ if $l > m+n,$ $a_5=a_6=a_9=a_{11}=0$ if  either $n < m$ or $ m < l.$

			\item $\mathbb{Z}_2[x,y,w,z]/<x^{n+1}, I_j, w^2,z^2,wz,x^{l-m+1}y,x^{l-m+1}w>_{1\leq j \leq 3}$, \\ where deg $x=1,$ deg $y=m,$ deg $w=n+m~\&$ deg $z=m+l;$
			$I_1=y^2+a_1x^{m-n}w,$
			$I_2=yw+a_2x^{m+n-l}z,$ and
			$I_3=yz+a_3x^nz,$ 
			$a_i \in \mathbb{Z}_2,$ $1\leq i \leq 3;$ $a_1=0$ if $2m>n+l; a_3=0 $ if $n < m. $

		\end{enumerate}

	\end{theorem}
	\begin{proof}
		If $d_{r_1}(1\otimes a)\neq 0,$ then $r_1=n+1$ and $d_{n+1}(1\otimes a)=t^{n+1}\otimes 1.$ For the remaining differentials, the following four cases are possible:  (i) $d_{r_2}(1\otimes b)=d_{r_3}(1\otimes c)= 0,$ (ii) $d_{r_2}(1\otimes b) \neq 0 ~\&~ d_{r_3}(1\otimes c)=0,$ (iii) $d_{r_2}(1\otimes b)= 0 ~\&~ d_{r_3}(1\otimes c) \neq 0$ and (iv) $ d_{r_2}(1\otimes b)\neq 0 ~ \& ~ d_{r_3}(1\otimes c) \neq 0 .$
		
		\noindent {\bf Case (i):} $d_{r_2}(1\otimes b)=0~\&$ $d_{r_3}(1\otimes c)=0.$\\
		In this case, we get  $d_{n+1}(1\otimes ab)=t^{n+1}\otimes b $, $d_{n+1}(1\otimes ac)=t^{n+1}\otimes c$ and $d_{n+1}(1\otimes abc)=t^{n+1}\otimes bc.$ So, $d_r=0$ for all $r>n+1.$ Thus,  $E_{n+2}^{*,*}=E_{\infty}^{*,*}.$ \\ If $n\leq m<l,$ then  $E_{\infty}^{p,q} \cong \mathbb{Z}_2$ for $0 \leq p \leq n$, $q=0,m,l,m+l,$ and  zero otherwise. \\ First, we consider $n<m<l.$ \\
		For $l>m+n,$ the cohomology groups $H^k(X_G)$ are given by 
		$$
		H^k(X_G)=
		\begin{cases}
			\mathbb{Z}_2  & j \leq k < n+j, j=0,m,l,m+l,\\
			0 & \mbox{otherwise,}
		\end{cases}
		$$ and, for $l\leq m+n,$ we have    
		$$
		H^k(X_G)=
		\begin{cases}
			\mathbb{Z}_2  & j \leq k < n+j, j=0,m+l; m \leq k < l; m+n < k\leq n+l,\\
			\mathbb{Z}_2 \oplus \mathbb{Z}_2 & l \leq k \leq m+n, \\
			0 & \mbox{otherwise.}
		\end{cases}
		$$ 
		The permanent cocycles  $t\otimes 1, 1 \otimes b,$ $1 \otimes c$ and $1 \otimes bc$ of $E_2^{*,*}$ determine the elements  $x\in E_{\infty}^{1,0}, u \in E_{\infty}^{0,m},$ $v \in E_{\infty}^{0,l}$ 
		and $s \in E_{\infty}^{0,m+l},$ respectively.  Thus, the total complex is given by \begin{center}
			Tot $E_{\infty}^{*,*} = \mathbb{Z}_2[x,u,v,s]/<x^{n+1},u^2+\gamma_1v,v^2,s^2,uv+\gamma_2s,us,vs>,$
		\end{center}where deg $x=1,$ deg $u= m,$ deg $v= l~\&$   deg $s= m+l$ and ${\gamma_1,\gamma_2 \in \mathbb{Z}_2},$ $\gamma_1=0$ if $l \neq 2m.$ Let $y \in H^m(X_G), w \in H^l(X_G)$ and $ z \in H^{m+l}(X_G)$ such that $i^*(y)=b,i^*(w)=c,$ and $i^*(z)=bc,$ respectively. Clearly,  $I_1= y^2+b_1x^{2m-l}w=0,$ $I_2= w^2+b_2x^{l-m}z=0 ~\&$ $I_3= yw+b_3z=0,$   where $b_i \in \mathbb{Z}_2, 1\leq i \leq 3,$ $b_1=0 $ if $2m>n+l$ or $l >2m $  and $b_2=0$ if $l>m+n.$   By Proposition \ref{mac}, the cohomology ring of the orbit space $X/G$ is given by $$ \mathbb{Z}_2[x,y,w,z]/<x^{n+1}, I_j, z^2,yz,wz>_{1\leq j \leq 3}, $$  where deg $x=1,$ deg $y=m$, deg $w=n+m$ and deg $z=m+l.$ This realizes possibility (1) by taking $a_i=0$ for $i=1,2,5,6$ and $8 \leq i \leq 11.$\\
		Now, we consider $n=m<l.$\\
		For $l>2n,$  the  cohomology groups $H^k(X_G)$ is given by 
		$$
		H^k(X_G)=
		\begin{cases}
			\mathbb{Z}_2  & j \leq k < n+j; n+j< k \leq 2n+j, j=0,l,\\
			\mathbb{Z}_2 \oplus \mathbb{Z}_2 & k=n,n+l,\\
			0 & \mbox{otherwise,}
		\end{cases}
		$$  and, for  $l \leq 2n,$ we get   
		$$
		H^k(X_G)=
		\begin{cases}
			\mathbb{Z}_2  & 0 \leq k < n;n+j<k<l+j, j=0,n; n+l < k \leq 2n+l,\\
			\mathbb{Z}_2 \oplus \mathbb{Z}_2 & k=n,n+l;l \leq k \leq 2n, \\
			0 & \mbox{otherwise.}
		\end{cases}
		$$ It is clear that the total complex is same as in the case when $n<m<l.$  Clearly, $I_1= y^2+b_1x^{2n-l}w+b_2x^ny=0,$ $I_2= w^2+b_3x^{l-n}z=0,$ $I_3= yw+b_4z+b_5x^nw=0$ and $I_4=yz+b_6x^nz=0$ where $ b_i \in \mathbb{Z}_2, 1\leq i \leq 6,$ $b_1=b_3=0$ if $l>2n.$ Hence, the  cohomology ring of the orbit space $X/G$ is given by $$\mathbb{Z}_2[x,y,w,z]/<x^{n+1}, I_j, z^2,wz>_{1\leq j \leq 4} ,$$ where deg $x=1,$ deg $y=n$, deg $w=l~\&$ deg $z=n+l.$ This realizes possibility (1) by taking $a_i=0$ for $i=1,5,6,9,11.$\\ 
		If $n\leq m=l,$ then $E_{\infty}^{p,q} \cong \mathbb{Z}_2$ for $0 \leq p \leq n$, $q=0,2m ;$  $E_{\infty}^{p,q}\cong \mathbb{Z}_2 \oplus \mathbb{Z}_2   $ for $0 \leq p \leq n$, $q=m, $ and   zero otherwise. For  $n<m=l,$ the cohomology  groups are given by 
		$$
		H^k(X_G)=
		\begin{cases}
			\mathbb{Z}_2  & j \leq k \leq n+j, j=0,2m,\\
			\mathbb{Z}_2 \oplus \mathbb{Z}_2 &  m \leq k \leq m+n,\\
			0 & \mbox{otherwise,}
		\end{cases}
		$$
		and, for $n=m=l,$ we have 
		$$
		H^k(X_G)=
		\begin{cases}
			\mathbb{Z}_2  & 0 \leq k < n, 2n <k \leq 3n,\\
			\mathbb{Z}_2 \oplus \mathbb{Z}_2 &  n < k < 2n,\\
			\mathbb{Z}_2 \oplus\mathbb{Z}_2 \oplus\mathbb{Z}_2 & k=n,2n,\\
			0 & \mbox{otherwise.}
		\end{cases}
		$$
		Thus,	the total complex  is given by \begin{center}
			Tot $E_{\infty}^{*,*} = \mathbb{Z}_2[x,u,v,s]/<x^{n+1},u^2+\gamma_1s,v^2+\gamma_2s,s^2,uv+\gamma_3s,us,vs>,$
		\end{center} where ${\gamma_1,\gamma_2,\gamma_3 \in \mathbb{Z}_2},$ deg $x=1,$ deg $u$ = deg $v$ = $m$ and  deg $s= 2m.$\\ If $n<m=l,$ then we get  $I_1= y^2+b_1z=0,$ $I_2= w^2+b_2z=0$  and $I_3= yw+b_3z=0,$ where $b_i\in \mathbb{Z}_2, 1\leq i \leq 3.$ Thus, the cohomology ring of the orbit space $X/G$ is given by $$ \mathbb{Z}_2[x,y,w,z]/<x^{n+1}, I_j, z^2,yz,wz>_{1\leq j \leq 3}, $$  where deg $x=1,$ deg $y$= deg $w=m~\&$ deg $z=2m.$ 	  This realizes possibility (1) by taking $a_i=0$ for $i=2,3,5,6 ~\&~ 8 \leq i \leq 11.$\\ 
		If $n=m=l$, then we get  $I_1=y^2+b_1z+b_2x^ny+b_3x^{n}w=0,I_2=w^2+b_4z+b_5x^ny+b_6x^nw=0, I_3=yw+b_7z+b_8x^nw+b_9x^ny=0,I_4 = yz+b_{10}x^nz=0$ and $ I_5=wz+b_{11}x^nz=0, b_i\in \mathbb{Z}_2, 1\leq i \leq 11.$ Thus, the  cohomology ring of the orbit space $X/G$ is given by $$ \mathbb{Z}_2[x,y,w,z]/<x^{n+1}, I_j, z^2>_{1\leq j \leq 5},$$ where deg $x=1,$ deg $y$ = deg $w=n$ and  $z=2n.$
		This realizes possibility (1).\\
		{\bf Case (ii)}  $d_{r_2}(1 \otimes b)\neq 0 ~~\&~~ d_{r_3}(1 \otimes c)=0.$\\ This case is possible only when $n=m.$ We have $d_{n+1}(1 \otimes a)=d_{n+1}(1\otimes b)= t^{n+1}\otimes 1.$ Consequently,  $d_{n+1}(1 \otimes ac)=d_{n+1}(1\otimes bc)= t^{n+1}\otimes c,$ $d_{n+1}(1 \otimes ab)= t^{n+1}\otimes (a+b)$ ~$\&$~ $d_{n+1}(1 \otimes abc)=t^{n+1}\otimes (bc+ac).$ Thus, $E_{n+2}^{*,*}=E_{\infty}^{*,*}.$ The elements  $t\otimes 1,1\otimes (a+b), 1 \otimes c$ and $1 \otimes c(a+b)$ are  the permanent cocycles, and the  cohomology groups and cohomology algebra of $X/G$ are  same as in case (i) when $n=m\leq l.$  \\
		{\bf Case (iii)}  $d_{r_2}(1 \otimes b)=0 ~~\&~~ d_{r_3}(1 \otimes c)\neq 0.$\\ As $d_{r_1}(1 \otimes a)\neq 0 ~~\&~~ d_{r_3}(1 \otimes c)\neq 0,$ it is easy to observe that this case is not possible when $l> m+n$ or $n<m=l.$ \\
		First, we consider $n \leq m <l\leq m+n.$ If $l-m<n,$ then $d_{l-m+1}$ must be nontrivial. So, we get $d_{l-m+1}(1\otimes c)=t^{l-m+1}\otimes b$ and $d_{l-m+1}(1\otimes ac)=t^{l-m+1}\otimes ab.$ Clearly, $d_r=0$  for $l-m+1<r<n+1.$ Since $d_{n+1}(1 \otimes a)=t^{n+1}\otimes 1,$ we get  $d_{n+1}(1 \otimes abc)=t^{n+1}\otimes bc.$ Thus, $E_{n+2}^{*,*}=E_{\infty}^{*,*},$  and $E_{\infty}^{p,q} \cong \mathbb{Z}_2$ for $0 \leq p \leq n$, $q=0,m+l; E_{\infty}^{p,q} \cong \mathbb{Z}_2$ for $0 \leq p \leq l-m$, $q=m,n+m,$ and  otherwise  zero. For $n <m <l<m+n,$ the cohomology groups are given by 
		\\
		$$
		H^k(X_G)=
		\begin{cases}
			\mathbb{Z}_2  & j \leq k \leq n+j, j=0,m+l; m+j \leq k \leq l+j, j=0,n,\\
			0 & \mbox{otherwise.}
		\end{cases}
		$$ For $n=m< l < 2n, $ we have 
		$$
		H^k(X_G)=
		\begin{cases}
			\mathbb{Z}_2  & 0 \leq k < n; n< k \leq l; 2n \leq k < n+l; n+l <k\leq 2n+l,\\
			\mathbb{Z}_2 \oplus \mathbb{Z}_2 & k=n, n+l, \\
			0 & \mbox{otherwise.}
		\end{cases}
		$$
		The permanent cocycles  $t\otimes 1, 1 \otimes b,$ $1 \otimes ab$ and $1 \otimes bc$ of $E_2^{*,*}$ determine the elements $x\in E_{\infty}^{1,0}, u \in E_{\infty}^{0,m},$ $v \in E_{\infty}^{0,n+m}$ 
		$\& ~s \in E_{\infty}^{0,m+l},$ respectively. 
		The total complex is given by \begin{center}
			Tot $E_{\infty}^{*,*} = \mathbb{Z}_2[x,u,v,s]/<x^{n+1},u^2+\gamma_1v,v^2,s^2,uv,us,vs,x^{l-m+1}u,x^{l-m+1}v>,$
		\end{center} where  deg $x=1,$ deg $u$ = $m$, deg $v$ = $n+m~\&$  deg $s= m+l,$ $ \gamma_1 \in \mathbb{Z}_2,\gamma_1=0$ if $n< m.$ Let $ y \in H^m(X_G), w \in H^{n+m}(X_G)$ and $ z \in H^{m+l}(X_G)$ such that $ i^*(y)=b,i^*(w)=ab,$ and $i^*(z)=bc,$ respectively.  Thus, for  $n\leq m<l<n+m,$ we have 	$I_1=y^2+b_1x^{m-n}w=0,$ $I_2=yw+b_2x^{m+n-l}z=0,$ $I_3=yz+b_3x^nz=0,$ where $b_1,b_2,b_3\in \mathbb{Z}_2$ and 	  $b_1=0 $ if $2m>n+l,b_3=0$ if $n< m.$
		Thus, the  cohomology ring  of the orbit space $X/G$ is given by  $$ \mathbb{Z}_2[x,y,w,z]/<x^{n+1}, I_j, w^2,z^2,wz,x^{l-m+1}y,x^{l-m+1}w>_{1\leq j \leq 3},$$  where 
		deg $x=1,$ deg $y=m$, deg $w=n+m,$ and deg $z=m+l.$   This realizes the possibility (2).\\ Next, we consider $n\leq m<l=m+n.$ We must have  $d_{n+1}(1\otimes ab)=d_{n+1}(1 \otimes c)=t^{n+1}\otimes b.$ Consequently, $d_{n+1}(1\otimes ac)=t^{n+1}\otimes (c+ab)$ and $d_{n+1}(1\otimes abc)= t^{n+1}\otimes bc.$ So, we get $E_{n+2}^{*,*}=E_{\infty}^{*,*}.$ The elements $t\otimes 1, 1\otimes b, 1\otimes (c+ab)$ and $1\otimes bc$ are permanent cocycles, and the cohomology algebra of $X/G$ is the same as case (i) when $n \leq m<l$ and $l=m+n.$\\ Finally, we consider $n=m=l.$ We must have $d_{n+1}(1\otimes a)=d_{n+1}(1 \otimes c)=t^{n+1}\otimes 1.$ So, we get $E_{n+2}^{*,*}=E_{\infty}^{*,*}.$ The elements  $t\otimes 1, 1\otimes b, 1\otimes (a+c)$ and $1\otimes b(a+c)$ are permanent cocycles, and  cohomology algebra is the same as in case (i) when $n=m=l.$ \\
		{\bf Case (iv)} $d_{r_2}(1 \otimes b)\neq 0 ~~\&~~ d_{r_3}(1 \otimes c)\neq 0.$\\ In this case, we must have $n=m \leq l.$ First, we consider $n=m<l.$ Clearly, $l< 2n $ or $l\geq 3n$ are not possible. Now, suppose that $2n<l<3n$ and  $d_{l-2n+1}$ is nontrivial. So, we get $d_{l-2n+1}(1 \otimes c)=t^{l-2n+1}\otimes ab$ and  $d_r=0$ for $l-2n+1<r\leq n.$ As $d_{n+1}(1\otimes a)=d_{n+1}(1 \otimes b)=t^{n+1}\otimes 1,$ we have $d_{n+1}(1 \otimes ab)=t^{n+1}\otimes (a+b).$ Thus, we get $0=d_{n+1}\{(t^{l-2n}\otimes ab)(t\otimes 1)\}=t^{l-n+1}\otimes (a+b),$ which is not possible. So, we must have $l=2n.$ We get $d_{n+1}(1\otimes a)=d_{n+1}(1 \otimes b)=t^{n+1}\otimes 1 ~\& ~ d_{n+1}(1\otimes ab)=d_{n+1}(1 \otimes c)=t^{n+1}\otimes (a+b).$ Consequently, $E_{n+2}^{*,*}=E_{\infty}^{*,*}.$ The elements  $t\otimes 1 , 1\otimes (a+b), 1 \otimes (ab+c)$ and $1 \otimes (ac+bc)$ are permanent cocycles, and the cohomology algebra of $X/G$ is the same as in  case (i) when $n=m < l=2n.$\\  Next, we consider $n=m=l.$ We have $d_{n+1}(1\otimes a)=d_{n+1}(1 \otimes b)= d_{n+1}(1 \otimes c)=t^{n+1}\otimes 1.$ So, we get $E_{n+2}^{*,*}=E_{\infty}^{*,*}.$ The elements $t \otimes 1 ,1\otimes (a+b), 1 \otimes (a+c)$ and $1 \otimes (ab+ac+bc)$ are permenent cocycles, and
		the cohomology algebra of $X/G$ is the same as in case (i) when $n=m=l.$  
	\end{proof}
	\begin{theorem} \label{thm 3.7}
		Let $G=\mathbb{Z}_2$ act freely on a finitistic space $X \sim_2 \mathbb{S}^n \times \mathbb{S}^m \times \mathbb{S}^l, $ where $n\leq m \leq l.$ If $d_{r_1}(1\otimes a)=0 $ and $d_{r_2}(1\otimes b) \neq 0,$ then $H^*(X/G)$ is isomorphic to one of the following graded commutative algebras:
		\begin{enumerate}
			
			\item  $\mathbb{Z}_2[x,y,w,z]/<Q(x), I_j, z^2,x^{m-n+1}y,x^{m-n+1}z>_{1\leq j \leq 5},$ \\ where deg $x=1,$ deg $y=n,$ deg $w=l~\&$ deg $z=n+l;$
			$I_1=y^2+a_1x^{2n}+a_2x^{2n-l}w+a_3x^{n}y,I_2=w^2+a_4x^{l}w+a_5x^{l-n}z+a_6x^{2l}, I_3=yw+a_7z+a_8x^{n+l}+a_9x^nw, I_4 = yz+a_{10}x^nz+a_{11}x^{2n}w+a_{12}x^{2n+l}~\&~ I_5=wz+a_{13}x^{n+m}w,$ $a_i \in \mathbb{Z}_2,$ $1\leq i \leq 13;$  $a_2=0$ if $l>2n,$ $a_3=a_{10}=0$ if $m<2n~\&~a_5=0$ if $m<l;$   $Q(x)=x^{n+m+j'+1},$ $j'=0 ~\&$ $l.$ \\ If $j'=0,$ then  $ a_4=0$ if $l>n+m,$ $a_{8}=a_{13}=0$ if $m<l$ and $a_6=a_{12}=0.$ \\
			If $j'=l,$ then $a_9=0$ if $ m<l,$ and  $a_i=0$ for $i=4,11,13.$ 
			\item  $\mathbb{Z}_2[x,y,w,z]/<x^{m+1}, I_j, z^2>_{1\leq j \leq 5},$\\ where 	deg $x$=1, deg $y=n$, deg $w=l~\&$ deg $z=n+l;$
			$I_1=y^2+a_1x^ny+a_2x^{2n-l}w+a_3x^{2n}+a_4z,I_2=w^2+a_5x^{l-n}z+a_6x^mw+a_7x^ny, I_3=yw+a_8z+a_9x^nw+a_{10}x^my, I_4 = yz+a_{11}x^nz+a_{12}x^{2n}w~\&~ I_5=wz+a_{13}x^mz,$ $a_i \in \mathbb{Z}_2$, $1\leq i \leq 13;$  $a_2=0$ if $l >2n,$ $a_3=0$ if $m < 2n ,$ $a_4=a_7=0$ if $n < m$ or $ m < l,$ $a_6=a_{10}=0$ if $ m< l$ and $a_5=0$ if  $l>m+n,a_{12}=0$ if $m>2n.$
			
			\item  $\mathbb{Z}_2[x,y,w,z]/<Q(x), I_j, z^2,wz,x^{m-n+1}y,x^{m-n+1}z,a_0x^{l-m-n+1}w>_{1\leq j \leq 4},$\\
			where deg $x=1,$ deg $y=n$, deg $w=n+m~\&$ deg $z=n+l;$ $I_1=y^2+a_1x^{2n}+a_2x^ny,I_2=w^2+a_3x^{n+m}w+a_4x^{2n+2m}+a_5x^{2m+n-l}z, I_3= yw+a_6x^nw+a_7x^{2n+m}+a_8x^{n+m-l}z~\&~ I_4 = yz+a_9x^{n+l-m}w+a_{10}x^nz+a_{11}x^{2n+l},$ $a_i\in \mathbb{Z}_2, 0 \leq i \leq 11;$ $ a_2=0$ if $m<2n;$  $Q(x)=x^{l+j'+1}, j'=0 ~\mbox{or}~ n+m.$ \\ If $j'=0,$ then $a_1=0$ if $l<2n,$ $a_3=0$ if $l<n+m$ or $m=l,$ $a_4=0$ if $l<2n+2m$ or $m=l,a_5=0$ if $l>m+2n$ or  $m=l,a_7=0$ if $l<2n+m,$ $a_8=0$ if  $l >m+n ,$  $a_{10}=0$ if $m>2n$ and $a_0=a_{11}=0.$\\
			If $j'=n+m ,$ then $a_{10}=0$ if $m<2n,$ $a_3=0$ if $l<2n+2m,$ $a_6=0$ if $l<2n+m,$ and $a_5=a_8=a_9=0~ \&~ a_0 =1.$ 	  
			\item 	$\mathbb{Z}_2[x,y,w,z]/<x^{m+1}, I_j,w^2,
			z^2,wz,x^{l-n+1}y,x^{l-n+1}w>_{1\leq j \leq 3},$\\ where deg $x=1,$ deg $y=n$, deg $w=n+m~\&$ deg $z=n+l;$ 
			$I_1=y^2+a_1x^{2n}+a_2x^{n}y +a_3w,I_2=yw +a_4x^{n}w+a_5x^{m+n-l}z~\&~ I_3= yz+a_6x^nz+a_7x^{l+n-m}w,$ $a_i\in \mathbb{Z}_2, 1 \leq i \leq 7;$  $a_1=a_7=0$ if $m<2n,$ $ a_2=a_4=0$ if  $l<2n,$ $a_3=0$ if $n < m.$
		\end{enumerate}
	\end{theorem}
	\begin{proof}
		If $d_{r_1}(1\otimes a)=0~\&~d_{r_2}(1\otimes b)\neq 0,$ then either $r_2=m-n+1$ or $r_2=m+1.$ 
		In this theorem, we consider two cases: (i) $d_{r_3}(1\otimes c)=0$ and (ii)  $d_{r_3}(1\otimes c)\neq 0.$\\
		{\bf Case (i):} $d_{r_3}(1\otimes c)=0.$\\
		First, Suppose that   $r_2=m-n+1.$ We must have $n<m\leq l,$ and  $d_{m-n+1}(1\otimes b)=t^{m-n+1}\otimes a.$ So, $d_{m-n+1}(1\otimes bc)=t^{m-n+1}\otimes ac.$ Assume that   $d_{m+n-l+1}$ is nontrivial. Then, we have $d_{m+n-l+1}(1\otimes ab)=t^{n+m-l+1}\otimes c.$ As $G$ acts freely on $X,$  we must have  $d_{m+n+l+1}(1\otimes abc)=t^{n+m+l+1}\otimes 1.$ Thus, $E_{n+m+l+2}^{*,*}=E_{\infty}^{*,*}$ and $E_{\infty}^{p,q} \cong \mathbb{Z}_2$ for $0 \leq p \leq n+m-l$, $q=l;0 \leq p \leq m-n$, $q=n,n+l;0 \leq p \leq n+m+l$, $q=0,$ and zero otherwise . For $n<m<l,$ the cohomology groups $H^k(X_G)$ are given by 
		\\
		$$
		\begin{cases}
			\mathbb{Z}_2  & 0 \leq k < n; m+j < k < l+j,j=0,n; m+l< k \leq m+n+l,\\
			\mathbb{Z}_2 \oplus \mathbb{Z}_2 & l \leq k \leq n+m; n+j \leq k \leq m+j, j=0,l,\\
			0 & \mbox{otherwise,}
		\end{cases}
		$$
		and, for $n<m=l,$ we have 
		$$
		H^k(X_G)=
		\begin{cases}
			\mathbb{Z}_2  & 0 \leq k < n; 2m< k \leq 2m+n,\\
			\mathbb{Z}_2 \oplus \mathbb{Z}_2 & n \leq k < m;m<k<n+m; n+m< k \leq 2m,\\
			\mathbb{Z}_2 \oplus \mathbb{Z}_2 \oplus \mathbb{Z}_2 & k=m,n+m,\\
			
			0 & \mbox{otherwise.}
		\end{cases}
		$$
		The permanent cocycles $t\otimes 1, 1 \otimes a,$ $1 \otimes c $ and $1 \otimes ac$ of $E_2^{*,*}$ determine the elements $x \in   E_{\infty}^{1,0}, u \in E_{\infty}^{0,n},$ $v \in E_{\infty}^{0,l}$ 
		and $ ~s \in E_{\infty}^{0,n+l},$ respectively.  The total complex Tot$ E_{\infty}^{*,*}$ is given by  $${\mathbb{Z}_2[x,u,v,s]}/ <x^{n+m+l+1}, u^2+\gamma_1v,v^2,s^2,uv+\gamma_2s,us,vs,x^{m-n+1}u,x^{m-n+1}s, x^{m+n-l+1}v>,$$ where deg $x=1,$ deg $u = n$, deg $v = l~\&$  deg $s= n+l$ and  $\gamma_1,\gamma_2 \in \mathbb{Z}_2,$ $\gamma_1=0$ if $l\neq 2n.$  Let $ y \in H^n(X_G), w \in H^{l}(X_G)$ and $ z \in H^{n+l}(X_G)$ such that $i^*(y)=a,i^*(w)=c$ and $i^*(z)=ac,$ respectively. Clearly, $I_1=y^2+b_1x^{2n}+b_2x^{2n-l}w+b_3x^{n}y=0,$ $I_2=w^2+b_4x^{2l}+b_5x^{m-n}z=0 ,$ $ I_3=yw+b_6z+b_7x^{n+l}+b_8x^{n}w=0,$ and $I_4=yz+b_9x^nz+b_{10}x^{2n+l},$ $b_i \in \mathbb{Z}_2, 1 \leq i \leq 10; $   $b_2=0$ if $l>2n, b_3=0$ if $m<2n,$  $b_5=b_8=0$ if $m<l,$ and $ b_9=0$ if $m<2n$.  Then, the  cohomology ring of the orbit space $X/G$ is given by $$ { \mathbb{Z}_2[x,y,w,z]}/{<x^{n+m+l+1}, I_j, z^2,wz,x^{m-n+1}y,x^{m-n+1}z,x^{m+n-l+1}w>_{1\leq j \leq 4.}}$$  where deg $x=1,$ deg $y=n,$ deg $w=l$ and deg $z=n+l.$ This realizes possibility (1) when $j'=l.$\\
		If $d_{m+n-l+1}$ is trivial, then $d_{n+m+1}$ must be nontrivial. Therefore, we have $d_{n+m+1}(1\otimes ab)=t^{n+m+1}\otimes 1.$ Consequently, $d_{n+m+1}(1\otimes abc)=t^{n+m+1}\otimes c.$ Thus, $E_{n+m+2}^{*,*}=E_{\infty}^{*,*}, $ and  hence $E_{\infty}^{p,q} \cong \mathbb{Z}_2$ for $0 \leq p \leq n+m$, $q=0,l;0 \leq p \leq m-n$, $q=n,n+l,$ and zero otherwise . For $n<m<l\leq m+n,$ the cohomology groups $H^k(X_G)$ are given by 
		\\
		$$
		\begin{cases}
			\mathbb{Z}_2  & 0 \leq k < n;  m+j < k < l+j,j=0,n;m+l< k \leq m+n+l,\\
			\mathbb{Z}_2 \oplus \mathbb{Z}_2 &  n+j \leq k \leq m+j, j=0,l;l \leq k \leq n+m,\\
			0 & \mbox{otherwise,}
		\end{cases}
		$$
		and, for $l>m+n,$ we have 
		$$
		H^k(X_G)=
		\begin{cases}
			\mathbb{Z}_2  & 0 \leq k < n; m+j< k \leq n+m+j,j=0,l; l\leq k<n+l,\\
			\mathbb{Z}_2 \oplus \mathbb{Z}_2 & n+j\leq k \leq m+j,j=0,l,\\

			0 & \mbox{otherwise.}
		\end{cases}
		$$	
		For $n<m=l,$  the cohomology groups  are given by 
		$$
		H^k(X_G)=
		\begin{cases}
			\mathbb{Z}_2  & 0 \leq k < n; 2m< k \leq 2m+n,\\
			\mathbb{Z}_2 \oplus \mathbb{Z}_2 & n \leq k <m;m<k<n+m; n+m< k \leq 2m,\\
			\mathbb{Z}_2 \oplus \mathbb{Z}_2 \oplus \mathbb{Z}_2 & k=m,n+m,\\
			
			0 & \mbox{otherwise.}
		\end{cases}
		$$
		The permanent cocycles $t\otimes 1, 1 \otimes a,$ $1 \otimes c $ and $1 \otimes ac$ of $E_2^{*,*}$ determine the elements  $x\in   E_{\infty}^{1,0},  u \in E_{\infty}^{0,n},$ $v \in E_{\infty}^{0,l},$ and
		$s \in E_{\infty}^{0,n+l},$ respectively. Thus, the total complex Tot$ E_{\infty}^{*,*}$ is given by \begin{center}
			$\mathbb{Z}_2[x,u,v,s]/<x^{n+m+1},u^2+\gamma_1v,v^2,s^2,uv+\gamma_2s,us,vs,x^{m-n+1}u,x^{m-n+1}s>,$
		\end{center}  where deg $x=1,$ deg $u$ = $n$, deg $v$ = $l$ $\&$  deg $s= n+l$ and $\gamma_1,\gamma_2 \in \mathbb{Z}_2, \gamma_1=0 $ if $l\neq 2n.$ Let $y \in H^n(X_G), w \in H^{l}(X_G)$ and $ z \in H^{n+l}(X_G)$ such that $i^*(y)=a,i^*(w)=c$ and $i^*(z)=ac,$ respectively. Clearly, we have  $I_1=y^2+b_1x^{2n}+b_2x^{2n-l}w+b_3x^ny=0,$ $I_2=w^2+b_4x^{l}w+b_5x^{l-n}z=0,$ $I_3=yw+b_6z+b_7x^{n+m}+b_8x^{n}w=0, I_4=yz+b_9x^nz+b_{10}x^{2n}w=0$ and $I_5=wz+b_{11}x^{n+m}w=0,$ $b_i \in \mathbb{Z}_2, 1 \leq i \leq 11;$ $ b_2=0$ if $l>2n, b_3=0$ if $m>2n,$   $b_4=0$ if $l>n+m,$ $b_5=b_7=b_{11} =0 $ if $ m<l,$   and $b_9=0$ if $m<2n.$  Thus, the  cohomology ring of the orbit space $X/G$ is given by $${ \mathbb{Z}_2[x,y,w,z]}/{<x^{n+m+1}, I_j, z^2,x^{m-n+1}y,x^{m-n+1}z>_{1\leq j \leq 5}},$$ where deg $x=1,$ deg $y=n,$ deg $w=l$ and deg $z=n+l.$ This realizes possibility (1) when $j'=0.$\\ 
		Next, we consider $r_2=m+1.$ We have $d_{m+1}(1\otimes b)=t^{m+1}\otimes 1.$ So, we get  $d_{m+1}(1\otimes ab)=t^{m+1}\otimes a $, $d_{m+1}(1\otimes bc)=t^{m+1}\otimes c$ and $d_{m+1}(1\otimes abc)=t^{n+1}\otimes ac.$ Clearly, $d_r=0$ for all $r>m+1.$ Thus, we get $E_{m+2}^{*,*}=E_{\infty}^{*,*}.$ \\ If $n\leq m<l,$ then  $E_{\infty}^{p,q} \cong \mathbb{Z}_2$ for $0 \leq p \leq m$, $q=0,n,l,n+l,$ and zero otherwise.\\ For $l\leq m+n,$ the cohomology groups $H^k(X_G)$ are given by \\
		$$
		\begin{cases}
			\mathbb{Z}_2  & 0 \leq k < n; m+j < k < l+j, j=0,n;m+l< k\leq m+n+l,\\
			\mathbb{Z}_2 \oplus \mathbb{Z}_2 & n+j\leq k \leq m+j, j=0,l; l \leq k \leq n+m,\\
			0 & \mbox{otherwise.}
		\end{cases}
		$$ and, for $l > m+n, $ we have 
		$$
		H^k(X_G)=
		\begin{cases}
			\mathbb{Z}_2  & j \leq k < n+j, j=0,l; m+j < k \leq m+n+j;j=0,l,\\
			\mathbb{Z}_2 \oplus \mathbb{Z}_2 & n+j \leq k \leq m+j,j=0,l, \\
			0 & \mbox{otherwise.}
		\end{cases}
		$$
		The permanent cocycles $t\otimes 1, 1 \otimes a,$ $1 \otimes $c and $1 \otimes ac$ of $E_2^{*,*}$ determine the elements $ x\in   E_{\infty}^{1,0}, u \in E_{\infty}^{0,n},$ $v \in E_{\infty}^{0,l}$ and
		$s \in E_{\infty}^{0,n+l},$ respectively.  Thus, the total complex is given by \begin{center}
			Tot $E_{\infty}^{*,*} = \mathbb{Z}_2[x,u,v,s]/<x^{m+1},u^2+\gamma_1v,v^2,s^2,uv+\gamma_2s,us,vs>,$	\end{center} where deg $x=1,$ deg $u= n,$ deg $v= l,$   deg $s= n+l,$ and ${\gamma_1,\gamma_2 \in \mathbb{Z}_2},$ $\gamma_1=0$ if $l \neq 2n.$ 
		Let $y \in H^n(X_G), w \in H^l(X_G)$ and $ z \in H^{n+l}(X_G)$ such that $i^*(y)=a,i^*(w)=c,$ and $i^*(z)=ac,$ respectively. For $n\leq m<l,$ we have $I_1= y^2+b_1x^{2n-l}w+b_2x^ny+b_3x^{2n}=0,$ $I_2= w^2+b_4x^{l-n}z=0,$   $I_3= yw+b_5z+b_6x^{n}w=0,$ and $I_4=yz+b_7x^nz+b_8x^{2n}w=0,$ $b_i \in \mathbb{Z}_2, 0\leq i \leq 8;$ $b_1=0 $ if $l>2n,$ $b_3=0$ if $m <2n,$ $b_4=0$ if $l>m+n$ and $b_8=0$ if $m>2n.$   Thus, the  cohomology ring of the orbit space $X/G$ is given by $$ \mathbb{Z}_2[x,y,w,z]/<x^{m+1}, I_j, z^2,wz>_{1\leq j \leq 4} ,$$ where  deg $x=1,$ deg $y=n,$ deg $w=l,$ and  deg $z=n+l.$	   This realizes possibility (2) by taking $a_i=0$ for $i=4,6,7,10,13.$ \\ 
		If $n< m=l,$ then $E_{\infty}^{p,q} = \mathbb{Z}_2$ for $0 \leq p \leq m$, $q=0,n,m,n+m $ and zero otherwise. Further, if $n=m=l,$ then $E_{\infty}^{p,q}=\mathbb{Z}_2$ for $0 \leq p \leq n$, $q=0,2n ;$ $E_{\infty}^{p,q}=\mathbb{Z}_2 \oplus \mathbb{Z}_2   $ for $0 \leq p \leq n$, $q=n, $ and  zero otherwise. For  $n<m=l,$ the cohomology  groups are given by 
		$$
		H^k(X_G)=
		\begin{cases}
			\mathbb{Z}_2  & 0 \leq k < n; 2m<k\leq 2m+n,\\
			\mathbb{Z}_2 \oplus \mathbb{Z}_2 &  n\leq k < m; m<k<m+n;m+n<k\leq2m,\\
			\mathbb{Z}_2 \oplus \mathbb{Z}_2 \oplus \mathbb{Z}_2 & k=m,m+n,\\
			0 & \mbox{otherwise.}
		\end{cases}
		$$
		and, for $n=m=l,$ we have 
		$$
		H^k(X_G)=
		\begin{cases}
			\mathbb{Z}_2  & 0 \leq k < n; 2n <k \leq 3n,\\
			\mathbb{Z}_2 \oplus \mathbb{Z}_2 &  n < k < 2n,\\
			\mathbb{Z}_2 \oplus\mathbb{Z}_2 \oplus\mathbb{Z}_2 & k=n,2n\\
			0 & \mbox{otherwise.}
		\end{cases}
		$$
		It is clear that for  $n<m=l,$ the total complex  is the same as in case when $n\leq m <l.$ For $n=m=l,$  the total complex is given by 
		\begin{center}
			Tot $E_{\infty}^{*,*} = \mathbb{Z}_2[x,u,v,s]/<x^{n+1},u^2+\gamma_1s,v^2+\gamma_2s,s^2,uv+\gamma_3s,us,vs>,$
		\end{center} where  deg $x=1,$ deg $u=$ deg $v=$ $m,$  deg $s= 2m,$ 
		and ${\gamma_1,\gamma_2 \in \mathbb{Z}_2}.$ For $n\leq m=l,$ we have $I_1= y^2+b_1x^{2n-m}+b_2x^{n}y+b_3x^{2n}+b_4z=0,$ $I_2= w^2+b_5x^{m-n}z+b_6x^{m}w+b_7x^ny=0,$ $I_3= yw+b_8z+b_9x^{n}w+b_{10}x^{m}y=0, I_4=yz+b_{11}x^nz=0 ~\&~ I_5=wz+b_{12}x^{m}z=0, b_i \in \mathbb{Z}_2, 1\leq i \leq 12;$ $b_1=0$ if $m>2n,$ $b_3=0$ if $m<2n$ and $b_4=b_{7}=0$ if $n<m.$ Thus, the  cohomology ring of the orbit space $X/G$ is given by $$ \mathbb{Z}_2[x,y,w,z]/<x^{m+1}, I_j, z^2>_{1\leq j \leq 5}, $$  where 	deg $x=1,$ deg $y=n,$ deg $w=m,\&$ deg $z=m+n.$  This realizes possibility (2).\\
		{\bf Case (ii):}  $d_{r_3}(1\otimes c)\neq 0.$\\ 
		First, suppose that $r_2=m-n+1.$ Then, we must have  $n <m \leq l,$ and  $d_{m-n+1}(1\otimes b)=t^{m-n+1}\otimes a.$ Consequently, $d_{m-n+1}(1\otimes bc)=t^{m-n+1}\otimes ac.$ In this case, we have $r_3=l-m-n+1$ or $l+1.$ \\If  $r_3=l+1,$ then $d_{l+1}(1\otimes c) = t^{l+1}\otimes 1$ and $d_{l+1}(1\otimes abc) = t^{l+1}\otimes ab.$ Thus, $E_{l+2}^{*,*}=E_{\infty}^{*,*}.$ If $n<m<l,$  then $E_{\infty}^{p,q}\cong \mathbb{Z}_2,$ for $0\leq p \leq m-n,q=n,n+l; 0 \leq p \leq l,q=0,n+m,$ and  zero otherwise. \\ For $l< m+n,$ the cohomology groups $H^k(X_G)$ are given by 
		$$
		\begin{cases}
			\mathbb{Z}_2  & 0 \leq k < n; m< k\leq l;m+n \leq k < n+l;m+l<k\leq m+n+l,\\
			\mathbb{Z}_2 \oplus \mathbb{Z}_2 & n+j\leq k \leq m+j, j=0,l,\\
			0 & \mbox{otherwise},
		\end{cases}
		$$  and, for $l \geq m+n, $ we have 
		$$
		H^k(X_G)=
		\begin{cases}
			\mathbb{Z}_2  & 0 \leq k < n; j < k < n+j,j=m,l; m+l<k\leq n+m+l,\\
			\mathbb{Z}_2 \oplus \mathbb{Z}_2 & n+j \leq k \leq m+j,j=0,l;m+n \leq k \leq l, \\
			0 & \mbox{otherwise.}
		\end{cases}
		$$
		If $n<m=l,$ then $E_{\infty}^{p,q}\cong \mathbb{Z}_2,$ for $0\leq p \leq m,q=0;0\leq p \leq m-n,q=n;m-n<p\leq m , q=n+m, $ and $E_{\infty}^{p,q}\cong \mathbb{Z}_2\oplus \mathbb{Z}_2,$ for $ 0 \leq p \leq m-n,q=n+m,$ and zero otherwise.\\
		Thus, the	cohomology groups of $X_G$ are given by 
		$$
		H^k(X_G)=
		\begin{cases}
			\mathbb{Z}_2  & 0 \leq k < n; 2m<k\leq 2m+n,\\
			\mathbb{Z}_2 \oplus \mathbb{Z}_2 & n+j\leq k \leq m+j, j=0,m,\\
			0 & \mbox{otherwise},
		\end{cases}
		$$ 
		The permanent cocycles $t\otimes 1, 1 \otimes a,$ $1 \otimes ab$ and $1 \otimes ac$ of $E_2^{*,*}$ determine the elements $x\in E_{\infty}^{1,0}, u \in E_{\infty}^{0,n},$ $v \in E_{\infty}^{0,n+m}$ and
		$s \in E_{\infty}^{0,n+l},$ respectively.  Thus, for $n<m\leq l,$ the total complex is given by \begin{center}
			Tot $E_{\infty}^{*,*} = \mathbb{Z}_2[x,u,v,s]/<x^{l+1},u^2,v^2+\gamma_1s,s^2,uv+\gamma_2s,us,vs,x^{m-n+1}u,x^{m-n+1}s>,$
		\end{center} where deg $x=1,$ deg $u= n,$ deg $v= n+m,$  deg $s= n+l$ and ${\gamma_1, \gamma_2  \in \mathbb{Z}_2};$ $\gamma_1=0$ if $l \neq n+2m$ and  $\gamma_2=0$ if $l \neq n+m.$ Let $y \in H^n(X_G), w \in H^{n+m}(X_G)$ and $ z \in H^{n+l}(X_G)$ such that $i^*(y)=a,i^*(w)=c,$ and $i^*(z)=ac,$ respectively. Clearly,  we have $I_1=y^2+b_1x^{2n}+b_2x^ny=0,$ $I_2=w^2+b_3x^{n+m}w+b_4x^{2n+2m}+b_5x^{2m+n-l}z=0,$ $ I_3=yw+b_6x^{n}w+b_7x^{2n+m}+b_8x^{n+m-l}z=0,$ $I_4=yz+b_9x^{n+l-m}w+b_{10}x^{n}z=0,b_i\in \mathbb{Z}_2, 1\leq i \leq 10;$ $ b_1=0 ~\mbox{if}~ l<2n,b_2=0$ if $ m<2n;$ $ b_3=0 ~\mbox{if}~ l<n+m ~\mbox{or}~ m=l, b_4=0 ~\mbox{if}~ l<2n+2m ~\mbox{or}~ m=l,$  $ b_5=0$ if $l>m+2n$ or $m=l,$ $b_7=0 ~\mbox{if}~ l<2n+m,b_8=0 $ if $l>n+m$ and  $ b_{10}=0 ~\mbox{if}~ m>2n.
		$ Thus, the  cohomology ring of the orbit space $X/G$ is given by  $$\mathbb{Z}_2[x,y,w,z]/<x^{l+1}, I_j, z^2,wz,x^{m-n+1}y,x^{m-n+1}z>_{1\leq j \leq 4},$$ where deg $x=1,$  deg $y=n,$  deg $w=n+m$ and  deg $z=n+l.$
		This realizes possibility (3) when $j'=0.$\\
		If $r_3=l-m-n+1,$ then we must have  $l>n+m$ and $d_{l-n-m+1}(1\otimes c)=t^{l-n-m+1} \otimes ab.$ As $G$ acts freely on $X,$ we get  $d_{n+m+l+1}(1\otimes abc)= t^{n+m+l+1}\otimes 1.$ Thus, $E_{n+m+l+2}^{*,*}=E_{\infty}^{*,*}.$ We get $E_{\infty}^{p,q}\cong\mathbb{Z}_2,$ for  $0 \leq p \leq n+m+l, q=0;0\leq p \leq m-n,q=n,n+l; 0\leq p \leq l-m-n,q=n+m,$ and zero otherwise. Consequently, the cohomology groups are given by
		$$
		H^k(X_G)=
		\begin{cases}
			\mathbb{Z}_2  & 0 \leq k < n; j < k < n+j,j=m,l; m+l<k\leq n+m+l,\\
			\mathbb{Z}_2 \oplus \mathbb{Z}_2 & n+j \leq k \leq m+j,j=0,l;m+n \leq k \leq l, \\
			0 & \mbox{otherwise.}
		\end{cases}
		$$
		The permanent cocycles $t\otimes 1, 1 \otimes a,$ $1 \otimes ab$ and $1 \otimes ac$ of $E_2^{*,*}$  determine the elements $x\in E_{\infty}^{1,0}, u \in E_{\infty}^{0,n},$ $v \in E_{\infty}^{0,n+m}$ and 
		$s \in E_{\infty}^{0,n+l},$ respectively.  Then, the total complex Tot $E_{\infty}^{*,*}$ is given by \begin{center}
			$\mathbb{Z}_2[x,u,v,s]/<x^{n+m+l+1},u^2,v^2+\gamma_1s,s^2,uv,us,vs,x^{m-n+1}u,x^{m-n+1}s,x^{l-m-n+1}v>,$
		\end{center}   where deg $x=1,$ deg $u= n,$ deg $v= n+m,$  deg $s= n+l$ and ${\gamma_1  \in \mathbb{Z}_2},$ $\gamma_1=0$ if $l \neq n+2m.$ Let $y \in H^n(X_G), w \in H^{n+m}(X_G)$ and $ z \in H^{n+l}(X_G)$ such that $i^*(y)=a,i^*(w)=c,$ and $i^*(z)=ac,$ respectively. Clearly, $I_1=y^2+b_1x^{2n} +b_2x^ny=0,$$I_2=w^2+b_3x^{2n+2m}+b_4x^{n+m}w=0,$ $I_3=yw+b_5x^{2n+m}+b_6x^nw=0,$ $I_4=yz+b_7x^{2n+l}+b_8x^nz,b_i\in \mathbb{Z}_2, 1\leq i \leq  8;$ $b_2=0$ if $m<2n; $ $b_4=0$ if $l<2n+2m, $ $b_6=0$ if $l<2n+m,$ and $b_8=0$ if $m<2n.$ So, the cohomology ring $H^*(X_G)$ is given by $${\mathbb{Z}_2[x,y,w,z]}/{<x^{n+m+l+1}, I_j, z^2,wz,x^{m-n+1}y,x^{m-n+1}z,x^{l-m-n+1}w>_{1\leq j \leq 4}}$$ where deg $x=1,$  deg $y=n,$  deg $w=n+m$ and deg $x=n+l.$ This realizes possibility (3) when $j'=m+n.$\\
		Finally, suppose that  $r_2=m+1.$ Clearly, $r_3\neq l-m+1.$ Further, if $r_3=l-m-n+1,$ then we get $0=d_{m+1}\{(t^{l-m-n}\otimes ab)(t\otimes 1)\}=t^{l-n+2}\otimes a,$ which is not possible. So, $r_3=l-n+1$ or $l+1.$\\ First, we consider $r_3=l-n+1.$ If $m<l-n,$ then $d_{r_3}(1\otimes c)=0,$ which contradicts our hypothesis. So, we get $l-n\leq m.$ As $d_{l-n+1}$ is nontrivial, we get $d_{l-n+1}(1\otimes c)=t^{l-n+1}\otimes a$ and $d_{l-n+1}(1\otimes bc)=t^{l-n+1}\otimes ab.$ Also,  we have $d_{m+1}(1\otimes b)=t^{m+1}\otimes 1~\&~d_{m+1}(1\otimes abc)=t^{m+1}\otimes  ac.$ Thus, $E_{m+2}^{*,*}=E_{\infty}^{*,*}.$ If $n\leq m<l$ and $l-n<m,$  then  $E_{\infty}^{p,q}\cong \mathbb{Z}_2,$ for $0\leq p \leq m,q=0,n+l; 0\leq p \leq l-n, q=n,n+m,$ and zero otherwise. Thus,  the cohomology groups  $H^k(X_G)$ are given by 
		$$
		\begin{cases}
			\mathbb{Z}_2  & 0 \leq k < n; m< k\leq l;m+n \leq k < n+l;m+l<k\leq m+n+l,\\
			\mathbb{Z}_2 \oplus \mathbb{Z}_2 & n+j\leq k \leq m+j, j=0,l,\\
			0 & \mbox{otherwise.}
		\end{cases}
		$$ The permanent cocycles $t\otimes 1, 1 \otimes a,$ $1 \otimes ab$ and $1 \otimes ac$ of $E_2^{*,*}$  determine the elements $x\in   E_{\infty}^{1,0},$ $u \in E_{\infty}^{0,n},$ $v \in E_{\infty}^{0,n+m}$ and $s \in E_{\infty}^{0,n+l},$ respectively.  Then, the total complex Tot $E_{\infty}^{*,*}$ is given by \begin{center}
			$\mathbb{Z}_2[x,u,v,s]/<x^{m+1},u^2+\gamma_1v,v^2,s^2,uv,us,vs,x^{l-n+1}u,x^{l-n+1}v>,$ 
		\end{center}  where deg $x=1,$ deg $u= n,$ deg $v= n+m,$  deg $s= n+l$ and  ${\gamma_1  \in \mathbb{Z}_2},$ $\gamma_1=0$ if $ n<m.$  Let  $y \in H^n(X_G), w \in H^{n+m}(X_G)$ and $ z \in H^{n+l}(X_G)$ such that $i^*(y)=a,i^*(w)=c,$ and $i^*(z)=ac,$ respectively. We have $I_1=y^2+b_1x^{2n}+b_2x^{n}y +b_3w=0,I_2=yw +b_4x^{n}w+b_5x^{m+n-l}z=0,$ and $ I_3= yz+b_6x^nz+b_7x^{l+n-m}w=0,$ $b_i\in \mathbb{Z}_2, 1 \leq i \leq 7;$  $b_1=b_7=0$ if $m<2n,$ $b_3=0$ if $n < m$ and $ b_2=b_4=0$ if  $l<2n.$ Therefore, the cohomology ring $H^*(X_G)$ is given by	$$\mathbb{Z}_2[x,y,w,z]/<x^{m+1}, I_j,w^2,
		z^2,wz,x^{l-n+1}y,x^{l-n+1}w>_{1\leq j \leq 3}$$
		deg $x=1,$ deg $y=n,$ deg $w=n+m,$ and  deg $z=n+l.$  This realize possibility (4). \\
		Next, if $n\leq m<l$ and $l-n=m,$ then the cohomology groups and  cohomology algebra are same in the case (i) when $r_2=m+1$ and $l=m+n.$\\ If $n<m=l$ and $d_{l-n+1}(1\otimes c)\neq 0,$ then $l-n=m-n$ and cohomology groups and cohomology algebra are same as in the   case (ii) when $n< m=l.$\\ Next, if $n<m=l$ and $d_{l-n+1}(1\otimes c)= 0,$ then we must have $r_2=r_3=m+1$ and hence, $d_{m+1}(1\otimes b)=d_{m+1}(1 \otimes c)=t^{m+1}\otimes 1.$  The cohomology groups and cohomology algebra are same as in the case (i) when $r_2=m+1$ and $n< m=l.$\\ Finally, if $n=m=l,$ then we must have $r_2=r_3=n+1$ and hence, $d_{n+1}(1\otimes b)=d_{n+1}(1 \otimes c)=t^{n+1}\otimes 1.$ The cohomology groups and cohomology algebra are the same as in  the case (i) when $n=m=l.$ 
	\end{proof}
	\begin{theorem}\label{thm 3.8}
		Let $G=\mathbb{Z}_2$ act freely on a finitistic space $X \sim_2 \mathbb{S}^n \times \mathbb{S}^m \times \mathbb{S}^l, $ where $n\leq m \leq l.$ If $d_{r_1}(1\otimes a)=d_{r_2}(1\otimes b)=0 $ and $d_{r_3}(1\otimes c) \neq 0,$ then the  cohomology ring of the orbit space $X/G$ is isomorphic to one of the following graded commutative algebras:
		\begin{enumerate}
			\item \label{1} $\mathbb{Z}_2[x,y,w,z]/I,$ where $I$ is homogeneous ideal given by: \\ $I=<Q(x), I_j,x^{l-m-n+1}z,c_0x^{l+n-m+1}w,c_1x^{l+m-n+1}y,c_2x^{l+1}y,c_3x^{l+1}w>_{1\leq j \leq 6},$
			where deg $x=1,$ deg $y=n$, deg $w=m~\&$ deg $z=n+m$ and $ I_1 =y^2+a_1x^{2n}+a_2x^{n}y+a_3x^{2n-m}w+a_4z,
			I_2 =w^2+a_5x^{2m}+a_6x^{2m-n}y+a_7x^mw+a_8x^{m-n}z,
			I_3= z^2+a_9x^{2n+2m} +a_{10}x^{n+2m}y+a_{11}x^{2n+m}w+a_{12}x^{n+m}z,
			I_4 =yw+a_{13}x^{n+m}+a_{14}x^my+a_{15}x^{n}w+a_{16}z,
			I_5 =yz+a_{17}x^{2n+m}+a_{18}x^{n+m}y+a_{19}x^{2n}w+a_{20}x^{n}z~\&~$ $
			I_6 =wz+a_{21}x^{2m+n}+a_{22}x^{2m}y+a_{23}x^{n+m}w+a_{24}x^{m}z, a_i\in \mathbb{Z}_2, 1\leq i \leq 24;$  $a_4=0$ if $n<m, a_8=0$ if $l<2m,$ $a_{12}=0$ if $l<2m+2n,$   $a_{20}=0$ if $l<2n+m,$ and  $a_{24}=0$ if $l<n+2m;$   $Q(x)=x^{l+1+j'}, j'=n+m,m$ or $n.$\\
			If $j'=n+m,$ then either \{$c_0=c_1=1~\&~ c_2=c_3=0$ with $a_7=0$ if $n+l<2m,a_{10}=0$ if $l<2n+m,a_{11}=0$ if $l<n+2m~\&~a_{23}=0$ if $l<2m$\} or \{$c_0=c_1=0 ~\&~ c_2=c_3=1$ with $a_{10}=0$ if $l<2n+m,a_{11}=0$ if  $l<n+2m~\&~a_{22}=0$ if $l<2m$\}.\\
			If $j'=m,$ then $c_0=1~\&~ c_1=c_2=c_3=0$ with $a_7=0$ if $n+l<2m, a_9=0$ if $l<2n+m,a_{11}=0$ if $l<n+2m~\&~$$a_{23}=0$ if $l<2m.$\\
			If $j'=n,$ then $c_1=1~ \&~ c_0=c_2=c_3=0$ with $a_5=0$ if $n+l<2m,a_{9}=0$ if $l<2m+n,a_{10}=0$ if $l<2n+m~\&~$$a_{21}=0$ if $l<2m.$  
			\item $\mathbb{Z}_2[x,y,w,z]/<Q(x), I_j,c_0x^{m+l-n+1}y,x^{l-m+1}z,x^{l-m+1}w>_{1\leq j \leq 6},$\\ where deg $x=1,$ deg $y=n$, deg $w=m~\&$ deg $z=n+m$ and ${I_j}'s,  1\leq j \leq 6$ are same as in possibility $(\ref{1}),$ with $a_3=0$ if $l<2n, a_4=0$ if $n < m,a_7=a_{24}=0 $ if $l<2m,a_{8}=0$ if $n+l<2m,a_{11}=0$ if $l<2n+2m,a_{12}=a_{23}=0$ if $l<n+2m,a_{15}=a_{20}=0$ if $l<m+n~\&~a_{19}=0$ if $l<2n+m,$ and  $Q(x)=x^{m+l+j'+1},j'=0 $ or $n.$\\
			If $j'=0,$ then $c_0=0 $ with $a_9=0$ if $l<2n+m,a_{10}=a_{21}=0$ if $l<n+m~\&~ a_{17}=0$ if  $l<2n.$\\ If $j'=n,$ then $c_0=1 $ with $ a_9=a_{22}=0$ if $l<n+m,a_{10}=0$ if $l<2n+m$$~\&~ a_{18}=0$ if  $l<2n.$ 
			\item   $\mathbb{Z}_2[x,y,w,z]/<Q(x), I_j,c_0x^{n+l-m+1}w,x^{l-n+1}y,x^{l-n+1}z>_{1\leq j \leq 6},$\\ where deg $x=1,$ deg $y=n$, deg $w=m~\&$ deg $z=n+m$ and  ${I_j}'s, 1\leq j \leq 6$ are same as in possibility $(\ref{1}),$ with $a_2=a_{20}=0$ if $l<2n, a_4=0$ if $n < m,a_6=0 $ if $l<2m,a_{10}=0$ if $l<2n+2m,a_{14}=a_{24}=0$ if $l<n+m~\&$ $a_{18}=0$ if $l<2n+m$ and $Q(x)=x^{n+1+j'},j'=0 $ or $m.$ \\
			If $j'=0,$ then $c_0=0 $ with $ a_5=0$ if $n+l<2m,a_{9}=a_{12}=0$ if $l<2m+n,a_{11}=a_{17}=0$ if  $l<n+m,a_{21}=0$ if $l<2m~\&~ a_{22}=0$ if $l<2n+m.$\\  If $j'=m,$ then $c_0=1 $ with $a_{7}=0$ if $n+l<2m, a_9=a_{19}=0$ if $l<n+m,a_{11}=a_{22}=0$ if $l<2m+n,a_{12}=0$ if $l<2n+m~\&~ a_{23}=0$ if  $l<2m.$ 
			\item $\mathbb{Z}_2[x,y,w,z]/<x^{l+1}, I_j>_{1\leq j \leq 6},$ where deg $x=1,$ deg $y=n$, deg $w=m~\&$ deg $z=n+m$ and ${I_j}'s, 1\leq j \leq 6$ are same as in possibility $(\ref{1}),$ with $a_1=a_{19}=0$ if $l<2n,$ $a_4=0 $ if $ n<m, a_5=a_{22}=0$ if $l<2m,a_{6}=0$ if $n+l<2m, a_{9}=0$ if $l<2n+2m,$ $a_{10}=a_{21}=0$ if $l<2m+n,$ $a_{11}=a_{17}=0$ if $l<2n+m~\&~$$a_{12}=a_{13}=a_{18}=a_{23}=0$ if $l<n+m.$   
		\end{enumerate}
	\end{theorem}	
	\begin{proof}
		If $d_{r_1}(1\otimes a)=d_{r_2}(1\otimes b) = 0$ and  $d_{r_3}(1\otimes c) \neq 0 ,$ then we have following four cases: (i) $r_3=l-m-n+1,$ (ii) $r_3=l-m+1,$ (iii) $r_3=l-n+1,$ and (iv) $r_3=l+1.$\\
		{\bf Case (i):} $r_3=l-m-n+1.$\\ Clearly, $n \leq m <l$ and $n+m<l.$  We have  $d_{l-m-n+1}(1\otimes c)= t^{l-m-n+1} \otimes ab.$\\ First, we assume that $d_{n+l-m+1}$ is nontrivial. So, we have $d_{n+l-m+1}(1\otimes ac)= t^{n+l-m+1}\otimes b.$ Now, we have either $d_{m+l-n+1}(1\otimes bc)=0$ or $d_{m+l-n+1}(1\otimes bc)=t^{m+l-n+1}\otimes a .$ \\Let $d_{m+l-n+1}(1\otimes ac)=t^{m+l-n+1}\otimes a .$ As $G$ acts freely on $X,$ we must have  $d_{n+m+l+1}(1\otimes abc)=t^{n+m+l+1}\otimes 1.$ Thus, $E_{n+m+l+2}^{*,*}=E_{\infty}^{*,*}.$ For $n<m< l,$ we have $E_{\infty}^{p,q}\cong \mathbb{Z}_2,$ $ 0 \leq p \leq n+m+l,q=0; 0 \leq p \leq m+l-n,q=n; 0 \leq p \leq n+l-m,q=m 
		$ and $0 \leq p \leq l-m-n,q=n+m,$ and zero otherwise. For $n=m<l,$ we have $E_{\infty}^{p,q} \cong \mathbb{Z}_2,$  $ 0 \leq p \leq 2n+l,q=0$ and $ 0 \leq p \leq l-2n,q=2n; E_{\infty}^{p,q} \cong \mathbb{Z}_2 \oplus \mathbb{Z}_2,$  $ 0 \leq p \leq l,q=n,$ and  zero otherwise. \\ For $n<m<l,$ the cohomology groups are given by 
		\[
		H^k(X_G) =
		\begin{cases}
			\mathbb{Z}_2  & 0 \leq k < n; m+l < k \leq m+n+l,\\
			\mathbb{Z}_2 \oplus \mathbb{Z}_2 & n \leq k < m;n+l < k \leq m+l, \\
			\mathbb{Z}_2\oplus \mathbb{Z}_2\oplus \mathbb{Z}_2 & m \leq k < m+n; l< k \leq n+l,\\
			\mathbb{Z}_2\oplus \mathbb{Z}_2\oplus \mathbb{Z}_2\oplus \mathbb{Z}_2 & m+n \leq k \leq l,\\
			0 & \mbox{otherwise,}
		\end{cases}
		\]
		and, for $n=m<l,$ we have
		\[
		H^k(X_G) =
		\begin{cases}
			\mathbb{Z}_2  & 0 \leq k < n; n+l < k \leq 2n+l,\\
			\mathbb{Z}_2\oplus \mathbb{Z}_2\oplus \mathbb{Z}_2 & n \leq k < 2n; l< k \leq n+l,\\
			\mathbb{Z}_2\oplus \mathbb{Z}_2\oplus \mathbb{Z}_2\oplus \mathbb{Z}_2 & 2n \leq k \leq l,\\
			0 & \mbox{otherwise.}
		\end{cases}
		\]	The permanent cocycles $t\otimes 1, 1 \otimes a,$ $1 \otimes b$ and $1 \otimes ab$ of $E_2^{*,*}$ determine the elements   $x\in   E_{\infty}^{1,0},$ $u \in E_{\infty}^{0,n},$ $v \in E_{\infty}^{0,m}$  and $s \in E_{\infty}^{0,n+m},$ respectively. Thus, the total complex \\Tot $E_{\infty}^{*,*}=\mathbb{Z}_2[x,u,v,s]/I,$ where ideal $I$ is given by \\  $<x^{n+m+l+1},u^2+\gamma_1v+\gamma_2s,v^2+\gamma_3s,s^2,uv+\gamma_4s,us,vs,x^{l+m-n+1}u,x^{l+n-m+1}v,x^{l-m-n+1}s>$ where  deg $x=1,$ deg $u= n,$ deg $v= m,$  deg $s= n+m,$ and ${\gamma_i \in \mathbb{Z}_2}, 1\leq i \leq 4;$ $\gamma_1=0$ if $m \neq 2n,$ and $\gamma_2=\gamma_3=0$ if $n < m.$ Let $y \in H^n(X_G), w \in H^{m}(X_G)$ and $ z \in H^{n+m}(X_G)$ such that $i^*(y)=a,i^*(w)=b,$ and $i^*(z)=ab,$ respectively. Clearly,
		$I_1 =y^2+b_1x^{2n}+b_2x^{n}y+b_3x^{2n-m}w+b_4z =0,
		I_2 =w^2+b_5x^{2m}+b_6x^{2m-n}y+b_7x^mw+b_8x^{m-n}z =0,
		I_3 = z^2+b_9x^{2n+2m} +b_{10}x^{n+2m}y+b_{11}x^{2n+m}w+b_{12}x^{n+m}z =0,
		I_4 =yw+b_{13}x^{n+m}+b_{14}x^my+b_{15}x^{n}w+b_{16}z =0,
		I_5 =yz+b_{17}x^{2n+m}+b_{18}x^{n+m}y+b_{19}x^{2n}w+b_{20}x^{n}z =0~\&~
		I_6 =wz+b_{21}x^{2m+n}+b_{22}x^{2m}y+b_{23}x^{n+m}w+b_{24}x^{m}z =0, 
		$
		where $b_i \in \mathbb{Z}_2, 1 \leq i \leq 24;$  $ b_4=0$ if $n< m , b_7=0$ if $n+l<2m,$ $b_8=b_{23}=0$ if $l<2m , b_{10}=b_{20}=0$ if $l<2n+m, b_{11}=b_{24}=0$ if $l<n+2m,$ and $b_{12}=0$ if $l<2n+2m.$ Thus, the cohomology ring $H^*(X_G)$ is given by	$$\mathbb{Z}_2[x,y,w,z]/<x^{n+m+l+1}, I_j,
		x^{l-m-n+1}z,x^{l+n-m+1}w,x^{l+n-m+1}y>_{1\leq j \leq 6},$$
		where deg $x=1,$ deg $y=n$, deg $w=m,$ and deg $z=n+m.$ This realizes possibility (1) when $j'=n+m$ with $c_0=c_1=1$ and $c_2=c_3=0.$\\ 
		Now, let $d_{m+l-n+1}(1\otimes ac)=0.$ Then,  we must  have $d_{m+l+1}(1 \otimes bc)=t^{m+l+1}\otimes 1$ and $d_{m+l+1}(1 \otimes abc)=t^{m+l+1}\otimes a.$ Thus, $E_{m+l+2}^{*,*}=E_{\infty}^{*,*}.$ For $n<m<l,$ we have  $E_{\infty}^{p,q}\cong \mathbb{Z}_2,$  $ 0 \leq p \leq m+l,q=0,n; 0 \leq p \leq n+l-m,q=m$ and $ 0 \leq p \leq l-m-n,q=n+m,$ and zero otherwise. For $n=m<l,$ we have $E_{\infty}^{p,q}\cong \mathbb{Z}_2,$  $ 0 \leq p \leq n+l,q=0; 0 \leq p \leq l-2n,q=2n;l < p \leq n+l,q=n; E_{\infty}^{p,q}\cong \mathbb{Z}_2 \oplus \mathbb{Z}_2,$  $ 0 \leq p \leq l,q=n,$ and zero otherwise. Note that the cohomology groups are the same as above. The total complex  Tot $ E_{\infty}^{*,*}=\mathbb{Z}_2[x,u,v,s]/I,$ where $I$ is an ideal given by $$<x^{m+l+1},u^2+\gamma_1v+\gamma_2s,v^2+\gamma_3s,s^2,uv+\gamma_4s,us,vs,x^{l+n-m+1}v,x^{l-m-n+1}s>, $$ where deg $x=1,$ deg $u= n,$ deg $v= m~\&$   deg $s= n+m,$ and ${\gamma_i \in \mathbb{Z}_2}, 1\leq i \leq 4;$ $\gamma_1=0$ if $m \neq 2n,$ and $\gamma_2=\gamma_3=0$ if $n < m.$  The ideals ${I_j}'s, 1 \leq j \leq 6$ are also same as above with conditions: $b_{4}=0$ if $n<m,$ $b_{7}=0$ if $m>2n,$ $b_{8}=b_{23}=0$ if $l<2m,$ $b_{9}=0$ if $l<2n+m,$ $b_{11}=0$ if $l<n+2m,$ $b_{12}=0$ if $l<2n+2m,$ $b_{20}=0$ if $l<2n+m,$ and $b_{24}=0$ if $l<n+2m.$ Thus, the cohomology ring $H^*(X_G)$ is given by $$\mathbb{Z}_2[x,y,w,z]/<x^{m+l+1}, I_j,
		x^{l-m-n+1}z,x^{l+n-m+1}w>_{1\leq j \leq 6},$$
		where deg $x=1,$ deg $y=n$, deg $w=m,$ and  deg $z=n+m.$ This realizes possibility (1) when $j'=m.$ \\
		Now, assume that  $d_{n+l-m+1}$ is trivial. We have either $d_{l+1}(1\otimes ac)=0$ or  $d_{l+1}(1 \otimes ac)= t^{l+1}\otimes a.$\\ Let $d_{l+1}(1 \otimes ac)= t^{l+1}\otimes a.$ Then, $d_{l+1}(1 \otimes bc)= t^{l+1}\otimes b,$ and   $d_{n+m+l+1}(1 \otimes abc)= t^{n+m+l+1}\otimes 1.$ Thus, $E_{n+m+l+2}^{*,*}=E_{\infty}^{*,*}.$ For $n<m<l,$ we have  $E_{\infty}^{p,q}\cong  \mathbb{Z}_2,$  $ 0 \leq p \leq n+m+l,q=0; 0 \leq p \leq l,q=n,m$ and $ 0 \leq p \leq l-m-n,q=n+m,$ and zero otherwise. For $n=m<l,$ we have $E_{\infty}^{p,q}\cong \mathbb{Z}_2,$ $ 0 \leq p \leq 2n+l,q=0; 0 \leq p \leq l-2n,q=2n; E_{\infty}^{p,q}\cong \mathbb{Z}_2 \oplus \mathbb{Z}_2,$ $ 0 \leq p \leq l,q=n,$ and zero otherwise. The cohomology groups are same as above. The total complex Tot$E_{\infty}^{*,*}=\mathbb{Z}_2[x,u,v,s]/I,$ where $I$ is an ideal given by\begin{center}
			$<x^{n+m+l+1},u^2+\gamma_1v+\gamma_2s,v^2+\gamma_3s,s^2,uv+\gamma_4s,us,vs,x^{l+1}u,x^{l+1}v,x^{l-m-n+1}s>,$	
		\end{center}  where  deg $x=1,$ deg $u= n,$ deg $v= m~\&$   deg $s= n+m,$ and ${\gamma_i \in \mathbb{Z}_2}, 1\leq i \leq 4;$ $\gamma_1=0$ if $m \neq 2n,$ and $\gamma_2=\gamma_3=0$ if $n<m.$  The ideals ${I_j}'s; 1 \leq j \leq 6$ are same as above with  conditions: $b_4=0 $ if $n<m, b_8=0$ if $l<2m,b_{10}=b_{24}=0$ if $l<n+2m,$ $b_{11}=b_{20}=0$ if $l<2n+m,$ $b_{12}=0 $ if $l<2n+2m,$  and $b_{22}=0$ if $l<2m.$  So, the cohomology ring $H^*(X_G)$ is given by $$\mathbb{Z}_2[x,y,w.z]/<x^{n+m+l+1}, I_j,
		x^{l-m-n+1}z,x^{l+1}y,x^{l+1}w>_{1\leq j \leq 6}$$
		where deg $x=1,$ deg $y=n$, deg $w=m,$ and deg $z=n+m.$ This realizes possibility (1) when $j'=n+m$ with $c_0=c_1=0$ and $c_2=c_3=1.$\\
		Now, let $d_{l+1}(1\otimes ac)=0.$ Then, we must have  $d_{m+l-n+1}(1\otimes bc)= t^{m+l-n+1}\otimes a.$ Consequently, we get   $d_{n+l+1}(1 \otimes ac)=t^{n+l+1}\otimes 1$ and  $d_{n+l+1}(1 \otimes abc)=t^{n+l+1}\otimes b.$ Thus, $E_{n+l+2}^{*,*}=E_{\infty}^{*,*}.$ For $n\leq m<l,$ we have  $E_{\infty}^{p,q}\cong \mathbb{Z}_2,$ $ 0 \leq p \leq n+l,q=0,m; 0 \leq p \leq m+l-n,q=n;$ $ 0 \leq p \leq l-m-n,q=n+m,$ and zero otherwise. The cohomology groups are same as above. The total complex Tot$E_{\infty}^{*,*}$ is given by $$\mathbb{Z}_2[x,u,v,s]/<x^{n+l+1},u^2+\gamma_1v+\gamma_2s,v^2+\gamma_3s,s^2,uv+\gamma_4s,us,vs,x^{l+m-n+1}u,x^{l-m-n+1}s>$$  where deg $x=1,$ deg $u= n,$ deg $v= m,$   deg $s= n+m,$ and ${\gamma_i \in \mathbb{Z}_2}, 1\leq i \leq 4;$ $\gamma_1=0$ if $m \neq 2n,$  and $\gamma_2=\gamma_3=0$ if $m \neq n.$ The ideals ${I_j}'s, 1 \leq j \leq 6$ are same as above  with  conditions: $b_4=0 $ if $n<m, b_5=0$ if $n+l<2m,b_{8}=b_{21}=0$ if $l<2m,b_{9}=b_{24}=0$ if $l<n+2m,$ $b_{10}=0$ if $l<2n+m,$  $b_{12}=0$ if $l<2n+2m,$ and $b_{20}=0$ if $l<2n+m.$ Thus, the cohomology ring $H^*(X_G)$ is given by $$\mathbb{Z}_2[x,y,w,z]/<x^{n+l+1}, I_j,
		x^{l-m-n+1}z,x^{m+l-n+1}y,>_{1\leq j \leq 6}$$
		where deg $x=1,$ deg $y=n$, deg $w=m,$ deg $z=n+m.$ This realizes possibility (1) when $j'=n.$\\
		{\bf Case (ii):} $r_3=l-m+1.$\\ Clearly,  $n \leq m <l.$ We have   $d_{l-m+1}(1\otimes c)= t^{l-m+1} \otimes b $ and $d_{l-m+1}(1\otimes ac)= t^{l-m+1} \otimes ab. $ Now, we have either $d_{m+l-n+1}(1\otimes bc)=0$ or $d_{m+l-n+1}(1 \otimes bc)= t^{m+l-n+1} \otimes a.$\\ First, let $d_{m+l-n+1}(1 \otimes bc)= t^{m+l-n+1} \otimes a.$  It is clear that  $ d_{n+m+l+1}(1 \otimes abc)= t^{n+m+l+1}\otimes 1 .$ Thus, $E_{n+m+l+2}^{*,*}=E_{\infty}^{*,*}.$ For $n<m<l,$ we have  $E_{\infty}^{p,q}\cong\mathbb{Z}_2,$  $ 0 \leq p \leq n+m+l,q=0; 0 \leq p \leq m+l-n,q=n;$  $ 0 \leq p \leq l-m,q=m,n+m,$ and zero otherwise. For $n=m<l,$ we have $E_{\infty}^{p,q}\cong \mathbb{Z}_2,$  $ 0 \leq p \leq 2n+l,q=0; 0 \leq p \leq l-n,q=2n,l-n < p \leq l,q=n;  E_{\infty}^{p,q}\cong \mathbb{Z}_2 \oplus \mathbb{Z}_2,$  $ 0 \leq p \leq l-n,q=n,$ and zero otherwise.\\  For $n<m<l<n+m,$ the cohomology groups are given by 
		\[
		H^k(X_G)=
		\begin{cases}
			\mathbb{Z}_2  & 0 \leq k < n; m+l < k \leq m+n+l,\\
			\mathbb{Z}_2 \oplus \mathbb{Z}_2 & n \leq k < m;l<k<n+m;n+l < k \leq m+l, \\
			\mathbb{Z}_2 \oplus \mathbb{Z}_2 \oplus \mathbb{Z}_2 & m+j \leq k \leq l+j,j=0,n,\\
			0 & \mbox{otherwise,}
		\end{cases}
		\] and, for $n=m<l<2n,$ we have 
		\[
		H^k(X_G)=
		\begin{cases}
			\mathbb{Z}_2  & 0 \leq k < n; n+l < k \leq 2n+l,\\
			\mathbb{Z}_2 \oplus \mathbb{Z}_2 & l<k<2n, \\
			\mathbb{Z}_2 \oplus \mathbb{Z}_2 \oplus \mathbb{Z}_2 & n+j \leq k \leq l+j,j=0,n,\\
			0 & \mbox{otherwise.}
		\end{cases}
		\]  For $l \geq n+m,$ the cohomology groups are similar to  case (i). 
		The total complex Tot$E_{\infty}^{*,*} = \mathbb{Z}_2[x,u,v,s]/I ,$ where ideal $I$ is given by \begin{center}
			$	<x^{n+m+l+1},u^2+\gamma_1v+\gamma_2s,v^2+\gamma_3s,s^2,uv+\gamma_4s,us,vs,x^{l+m-n+1}u,x^{l-m+1}v,x^{l-m+1}s>,$
		\end{center}  where  deg $x=1,$ deg $u= n,$ deg $v= m,$   deg $s= n+m,$ and ${\gamma_i \in \mathbb{Z}_2},$ $ 1\leq i \leq 4;$ $\gamma_1=0$ if $m \neq 2n,$ and $\gamma_2=\gamma_3=0$ if $n<m.$ The ideals ${I_j}'s, 1 \leq j \leq 6$ are same as in case (i)  with  conditions: $b_3=b_{18}=0$ if $l<2n,$ $b_4=0 $ if $ n<m, b_7=b_{24}=0$ if $l<2m,b_{8}=0$ if $n+l<2m,b_{9}=b_{15}=b_{20}=b_{22}=0$ if $l<n+m,$ $b_{10}=b_{19}=0$ if $l<2n+m, b_{11}=0$ if $l<2n+2m,$ and $b_{12}=b_{23}=0$ if $l<n+2m.$ Thus, the cohomology ring $H^*(X_G)$ is given by $$\mathbb{Z}_2[x,y,w.z]/<x^{n+m+l+1}, I_j,
		x^{m+l-n+1}y,x^{l-m+1}w,x^{l-m+1}z>_{1\leq j \leq 6}$$
		where deg $x=1,$ deg $y=n$, deg $w=m,$ and deg $z=n+m.$ This realizes possibility (2) when $j'=n.$\\
		Now, let  $d_{m+l-n+1}(1\otimes bc)=0.$ Then,  we must have  $d_{m+l+1}(1 \otimes bc)=t^{m+l+1}\otimes 1$ and $d_{m+l+1}(1 \otimes abc)=t^{m+l+1}\otimes a.$ Thus, $E_{m+l+2}^{*,*}=E_{\infty}^{*,*}.$ For $n<m<l,$ we have  $E_{\infty}^{p,q}\cong \mathbb{Z}_2,$ if $ 0 \leq p \leq m+l,q=0,n; 0 \leq p \leq l-m,q=m,n+m,$ and  zero otherwise. For $n=m<l,$ we have $E_{\infty}^{p,q}\cong \mathbb{Z}_2,$ $ 0 \leq p \leq n+l,q=0; l-n < p \leq n+l,q=n;0 \leq p \leq l-n,q=n+m;  E_{\infty}^{p,q}\cong \mathbb{Z}_2 \oplus \mathbb{Z}_2,$ $ 0 \leq p \leq l-n,q=n,$ and zero otherwise. For $l<n+m,$ the cohomology groups are same as case (ii) when  $d_{m+l-n+1}$ is nontrivial, and  for $l \geq n+m,$  the cohomology groups are similar to case (i). 
		Thus, the total complex Tot $E_{\infty}^{*,*}=\mathbb{Z}_2[x,u,v,s]/I,$ where ideal $I$ is given by  \begin{center}
			$<x^{m+l+1},u^2+\gamma_1v+\gamma_2s,v^2+\gamma_3s,s^2,uv+\gamma_4s,us,vs,x^{l-m+1}v,x^{l-m+1}s>,$
		\end{center} where   deg $x=1,$ deg $u= n,$ deg $v= m,$   deg $s= n+m,$ and ${\gamma_i \in \mathbb{Z}_2},$ $ 1\leq i \leq 4;$ $\gamma_1=0$ if $m \neq 2n,$ and $\gamma_2=\gamma_3=0$ if $n<m.$ The ideals ${I_j}'s, 1 \leq j \leq 6$ are same as in case (i)  with  conditions: $b_3=b_{17}=0$ if $l<2n,b_4=0$ if $n<m, b_7=b_{24}=0$ if $l<2m,b_8=0$ if $n+l<2m,b_9=b_{19}=0$ if $l<2n+m, b_{10}=b_{15}=b_{20}=b_{21}=0$ if $l<m+n,b_{11}=0$ if $l<2n+2m,$ $b_{12}=b_{23}=0$ if $l<2m+n,$ and $b_{17}=0$ if $l<2n.$  Thus, the cohomology ring $H^*(X_G)$ is given by $$\mathbb{Z}_2[x,y,w,z]/<x^{m+l+1}, I_j,
		x^{l-m+1}w,x^{l-m+1}z>_{1\leq j \leq 6},$$
		where deg $x=1,$ deg $y=n,$ deg $w=m$ and  deg $z=n+m.$ This realizes possibility (2) when $j'=0.$\\
		{\bf Case (iii):}  $r_3=l-n+1.$\\Clearly, $n<l.$ We have  $d_{l-n+1}(1\otimes c)= t^{l-n+1} \otimes a $ and $d_{l-n+1}(1\otimes bc)= t^{l-n+1} \otimes ab.$ We have either  $d_{n+l-m+1}(1\otimes ac)=0$ or $d_{n+l-m+1}(1 \otimes ac)=t^{n+l-m+1}\otimes b.$\\  First, let   $d_{n+l-m+1}(1 \otimes ac)=t^{n+l-m+1}\otimes b.$ Then, we  must have $d_{n+m+l+1}(1 \otimes abc)= t^{n+m+l+1}\otimes 1.$ Thus, $E_{n+m+l+2}^{*,*}=E_{\infty}^{*,*}.$ For $n<m\leq l,$ we have  $E_{\infty}^{p,q}\cong \mathbb{Z}_2$ if $ 0 \leq p \leq n+m+l,q=0; 0 \leq p \leq l-n,q=n,n+m; 0 \leq p \leq n+l-m,q=m,$ and  zero otherwise. For $n=m<l$ we have $E_{\infty}^{p,q}\cong \mathbb{Z}_2$ if $ 0 \leq p \leq 2n+l,q=0; l-n < p \leq l,q=n;0 \leq p \leq l-n,q=2n; E_{\infty}^{p,q}\cong\mathbb{Z}_2 \oplus \mathbb{Z}_2$ if $ 0 \leq p \leq l-n,q=n,$ and zero otherwise. For $l \geq n+m,$ the cohomology groups are similar to case (i),  and  for $l<n+m,$ the cohomology groups are same as in  case (ii). Thus, the total complex Tot$E_{\infty}^{*,*}$ = $\mathbb{Z}_2[x,u,v,s]/I,$ where $I$ is given by \begin{center}
			$<x^{n+m+l+1},u^2+\gamma_1v+\gamma_2s,v^2+\gamma_3s,s^2,uv+\gamma_4s,us,vs,x^{l-n+1}u,x^{n+l-m+1}v,x^{l-n+1}s>,$
		\end{center}  where deg $x=1,$ deg $u= n,$ deg $v= m,$  deg $s= n+m,$ and ${\gamma_i \in \mathbb{Z}_2},$ $ 1\leq i \leq 4;$ $\gamma_1=0$ if $m \neq 2n,$ and $\gamma_2=\gamma_3=0$ if $n < m.$  Clearly, the ideals $I_j, 1 \leq j \leq 6$ are same as in case (i) with  conditions: $b_2=b_{20}=0$ if $l<2n,$ $b_4=0 $ if $ n<m, b_6=b_{23}=0$ if $l<2m,b_{7}=0$ if $n+l<2m,b_{9}=b_{14}=b_{19}=b_{24}=0$ if $l<n+m,b_{10}=0$ if $l<2n+2m,$ $b_{11}=b_{22}=0$ if $l<2m+n$ and $b_{12}=b_{18}=0$ if $l<2n+m.$    Thus, the cohomology ring $H^*(X_G)$ is given by $$\mathbb{Z}_2[x,y,w.z]/<x^{n+m+l+1}, I_j,
		x^{l-n+1}y,x^{n+l-m+1}w,x^{l-n+1}z>_{1\leq j \leq 6}$$
		where deg $x$=1, deg $y=n$, deg $w=m,$ and deg $z=n+m.$ This realizes possibility (3) when $j'=m.$\\
		Now, let   $d_{n+l-m+1}(1\otimes ac)=0.$ Then, we must have  $d_{n+l+1}(1 \otimes ac)=t^{m+l+1}\otimes 1$ and $d_{n+l+1}(1 \otimes abc)=t^{m+l+1}\otimes b.$ Thus, $E_{n+l+2}^{*,*}=E_{\infty}^{*,*}.$ For $n<m\leq l,$ we have  $E_{\infty}^{p,q}\cong \mathbb{Z}_2,$ if $ 0 \leq p \leq n+l,q=0,m; 0 \leq p \leq l-n,q=n,n+m,$ and  zero otherwise. For $n=m<l,$ we have $E_{\infty}^{p,q}\cong \mathbb{Z}_2,$ if $ 0 \leq p \leq n+l,q=0; l-n < p \leq n+l,q=n;0 \leq p \leq l-n,q=2n;  E_{\infty}^{p,q}\cong \mathbb{Z}_2 \oplus \mathbb{Z}_2$ if $ 0 \leq p \leq l-n,q=n,$ and zero otherwise. For $l \geq n+m,$ the cohomology groups are similar to case (i), and for $l<n+m,$ the cohomology groups are same as case (ii). Thus, the total complex Tot $E_{\infty}^{*,*}$ = $\mathbb{Z}_2[x,u,v,s]/I,$ where ideal $I$ is given by \begin{center}
			$<x^{n+l+1},u^2+\gamma_1v+\gamma_2s,v^2+\gamma_3s,s^2,uv+\gamma_4s,us,vs,x^{l-n+1}u,x^{l-n+1}s>,$
		\end{center} where  deg $x=1,$ deg $u= n,$ deg $v= m,$  deg $s= n+m,$ and ${\gamma_i \in \mathbb{Z}_2},$ $ 1\leq i \leq 4;$ $\gamma_1=0$ if $m \neq 2n,$ and $\gamma_2=\gamma_3=0$ if $ n<m.$  The ideals ${I_j}'s; 1 \leq j \leq 6$ are same as in case (i) with  conditions: $b_2=b_{20}=0$ if $l<2n,$ $b_4=0 $ if $ n<m, $ $b_{5}=0$ if $n+l<2m,$ $b_6=b_{21}=0$ if $l<2m,$ $b_{9}=b_{12}=0$ if $l<2m+n,$  $b_{10}=0$ if  $l<2n+2m,b_{11}=b_{14}=b_{17}=b_{24}=0$ if $l<n+m$ and $b_{18}=b_{22}=0$ if $l<2n+m.$   Thus, the cohomology ring $H^*(X_G)$ is given by $$\mathbb{Z}_2[x,y,w,z]/<x^{n+l+1}, I_j,
		x^{l-n+1}y,x^{l-n+1}z>_{1\leq j \leq 6},$$
		where deg $x=1,$ deg $y=n,$ deg $w=m$ and deg $z=n+m.$ This realizes possibility (3) when $j'=0.$\\
		{\bf Case (iv):}  $r_3=l+1.$\\ We have  $d_{l+1}(1\otimes c)= t^{l+1} \otimes 1. $  Consequently,   $d_{l+1}(1\otimes ac)= t^{l+1} \otimes a,$ $d_{l+1}(1\otimes bc)= t^{l+1} \otimes b$ and  $d_{l+1}(1\otimes abc)= t^{l+1} \otimes ab.$ Thus, $E_{l+2}^{*,*}=E_{\infty}^{*,*}.$ For $n<m\leq l ,$ we have  $E_{\infty}^{p,q}\cong \mathbb{Z}_2,$ where $ 0 \leq p \leq l,q=0,n,m,n+m,$ and  zero otherwise. For $n=m\leq l,$ we have $E_{\infty}^{p,q}\cong \mathbb{Z}_2,$  $ 0 \leq p \leq l,q=0,2n;  E_{\infty}^{p,q}\cong \mathbb{Z}_2 \oplus \mathbb{Z}_2,$  $ 0 \leq p \leq l,q=n,$ and  zero otherwise. For $l \geq n+m,$ the cohomology groups are similar to case (i), and  for $l<n+m,$ the cohomology groups are same as in case (ii). The total complex Tot $E_{\infty}^{*,*}$ = $\mathbb{Z}_2[x,u,v,s]/I,$ where ideal $I$ is given by \begin{center}
			$<x^{l+1},u^2+\gamma_1v+\gamma_2s,v^2+\gamma_3s,s^2,uv+\gamma_4s,us,vs>,$
		\end{center} where  deg $x=1,$ deg $u= n,$ deg $v= m,$  deg $s= n+m,$ and ${\gamma_i \in \mathbb{Z}_2},$ $ 1\leq i \leq 4;$ $\gamma_1=0$ if $m \neq 2n,$ and $\gamma_2=\gamma_3=0$ if $ n<m.$ The ideals ${I_j}'s; 1 \leq j \leq 6$ are same as in case (i) with conditions:
		$b_1=b_{19}=0$ if $l<2n,$ $b_4=0 $ if $ n<m, b_5=b_{22}=0$ if $l<2m,b_{6}=0$ if $n+l<2m, b_{9}=0$ if $l<2n+2m,$ $b_{10}=b_{21}=0$ if $l<2m+n,$ $b_{11}=b_{17}=0$ if $l<2n+m,$ and $b_{12}=b_{13}=b_{18}=b_{23}=0$ if $l<n+m.$ Thus, the cohomology ring $H^*(X_G)$ is given by $$\mathbb{Z}_2[x,y,w.z]/<x^{l+1}, I_j>_{1\leq j \leq 6},$$
		where deg $x=1,$ deg $y=n$, deg $w=m$ and deg $z=n+m.$ This realizes possibility (4).	
	\end{proof}
	\begin{example}
		An example of case (1) of Theorem \ref{thm 3.6} can be realized by considering  diagonal action of $G=\mathbb{Z}_2$ on $\mathbb{S}^n \times \mathbb{S}^m \times \mathbb{S}^l,$ where $\mathbb{Z}_2$ acts freely on $\mathbb{S}^n$ and trivially on both $\mathbb{S}^m$ and $\mathbb{S}^l.$  Then, $X/G \sim_2 \mathbb{RP}^n \times \mathbb{S}^m \times \mathbb{S}^l.$ This realizes case (1) by taking $a_7=1$ and $a_i =0$ for $i\neq 7.$ Similarly, case (2) of Theorem  \ref{thm 3.7}, with $a_8=1$ and $a_i=0$ for $i\neq 8,$  and case (4) of Theorem \ref{thm 3.8}, with $a_{16}=1$ and $a_i=0$ for $i\neq 16,$  can be realized.
	\end{example}
	\section{Some Applications on $\mathbb{Z}_2$-Equivarient Maps}
	In this section,  we derive the Borsuk-Ulam type results for free involutions on $X \sim_2 \mathbb{S}^n \times \mathbb{S}^m \times \mathbb{S}^l ,1\leq n \leq m \leq l.$ We determine the nonexistence of $\mathbb{Z}_2$-equivariant maps between  $ X$ and $ \mathbb{S}^{k},$ where $\mathbb{S}^k$ equipped with  the antipodal actions.\\
	\indent Recall that \cite{floyd} the index (respectively, co-index) of a $G$-space $X$ is the greatest integer $k$ (respectively, the lowest integer $k$)  such that there exists a $G$-equivariant map $\mathbb{S}^k \rightarrow X$ (respectively, $ X \rightarrow  \mathbb{S}^k$).\\
	\indent	By  Theorems proved in section 3, we get the largest integer  $s$ for which $w^s \neq 0$ is one of $ n,m,l,n+m,2n+l,n+l,m+l$ or $n+m+l,$ where $w \in H^{1}(X/G)$ is the  characteristic class of the principle $G$-bundle $G \hookrightarrow X \rightarrow X/G.$ We know that index$(X)\leq s$ \cite{floyd}. Thus, we have the following Result:
	
	\begin{proposition}
		Let $G=\mathbb{Z}_2$ act freely on $X \sim_2 \mathbb{S}^n \times \mathbb{S}^m \times \mathbb{S}^l,$  $ n \leq m \leq l.$ Then there does not exist $G-$equivarient map from $\mathbb{S}^d$ to $X,$ for $d > h,$  where $ h$ is one of $ n,m,l,n+m,2n+l,n+l,m+l$ or $n+m+l.$
	\end{proposition}
	Recall that the Volovikov's index $i(X)$ is the smallest integer $r\geq 2$ such that $d_r: E_r^{k-r,r-1} \rightarrow E_r^{k,0}$ is nontrivial for some $k,$ in the Leray-Serre spectral sequence of the Borel fibration  $ X \stackrel{i} \hookrightarrow X_G \stackrel{\pi} \rightarrow B_G$ \cite{yu}. Again, 	by  Theorems proved in section 3, we get $i(X)$ is $ n, m,l,m-n,l-m,l-n$ or $l-m-n.$
	By taking $Y=\mathbb{S}^{k}$ in Theorem 1.1 \cite{co}, we have

	\begin{theorem}
		Let $G=\mathbb{Z}_2$ act freely on a finitistic $X \sim_2 \mathbb{S}^n \times \mathbb{S}^m \times \mathbb{S}^l,$  $ n\leq m \leq l .$ Then, there is no $G$-equivariant map $f: X \rightarrow \mathbb{S}^{k}$  if $1\leq k < i(X)-1,$ where $i(X)$ is one of  $n,m,l,m-n,l-m,l-n$ or $l-m-n.$
	\end{theorem}

\end{document}